\documentclass{amsart}
\usepackage{amsmath}
\usepackage{paralist}
\usepackage{amsfonts}
\usepackage{amssymb}
\usepackage{amsthm}
\usepackage{amscd}
\usepackage{amsrefs}
\usepackage{float}
\usepackage{tikz}
\usepackage{graphicx}
\usepackage[colorlinks=true]{hyperref}
\hypersetup{urlcolor=blue, citecolor=red}
\usepackage{hyperref}

\newtheorem{theorem}{Theorem}[section]

\newtheorem{lemma}[theorem]{Lemma}
\newtheorem{proposition}[theorem]{Proposition}

\theoremstyle{definition}
\newtheorem{definition}[theorem]{Definition}
\newtheorem{remark}[theorem]{Remark}

\newcommand{\T}{\mathbb{T}}
\newcommand{\R}{\mathbb{R}}
\newcommand{\Z}{\mathbb{Z}}

\newcommand{\C}{\mathbb{C}}

\begin{document}
\title[Dispersion generalized Benjamin-Ono equations on the circle]
      {Local and global well-posedness of dispersion generalized Benjamin-Ono equations on the circle}

\author[R. Schippa]{Robert Schippa}
\address{Fakult\"at f\"ur Mathematik, Universit\"at Bielefeld, Postfach 10 01 31, 33501 Bielefeld, Germany}
\email{robert.schippa@uni-bielefeld.de}

\thanks{Financial support by the German Science Foundation (IRTG 2235) is gratefully acknowledged.}
\maketitle

\bigskip

\begin{abstract}
New local well-posedness results for dispersion generalized Benjamin-Ono equations on the torus are proved. The family of equations under consideration links the Benjamin-Ono and Korteweg-de Vries equation. For sufficiently strong dispersion global well-posedness in $L^2(\mathbb{T})$ is derived. 
\end{abstract}

\section{Introduction}
\label{section:Introduction}
In this article we prove new well-posedness results for the one-dimensional fractional Benjamin-Ono equation in the periodic case
\begin{equation}
\label{eq:GeneralizedFractionalBenjaminOnoEquation}
\left\{\begin{array}{cl}
\partial_t u + \partial_{x} D_x^{a} u &= u \partial_x u ,  \; (t,x) \in \mathbb{R} \times \mathbb{T}, \\
u(0) &= u_0 \in H^s(\mathbb{T}), \end{array} \right.
\end{equation}
where $\mathbb{T} = \mathbb{R}/(2\pi \mathbb{Z})$, $D_x = (-\Delta)^{1/2}$ and $1<a<2$ will be considered.\\
When we refer to local well-posedness in the following, we mean that the data-to-solution mapping $S_T^\infty: H^\infty(\T) \rightarrow C([0,T],H^\infty(\T))$ admits a continuous extension $S_T^s:H^s \rightarrow C([0,T],H^s)$ with $T=T(\Vert u_0 \Vert_{H^s})$ which can be chosen continuously on $\Vert u_0 \Vert_{H^s}$. The existence of $S_T^\infty:H^\infty \rightarrow C_T H^s$ follows from the classical energy method (cf. \cite{BonaSmith1975,AbdelouhabBonaFellandSaut1989}).\\
Note that in non-negative Sobolev spaces the quadratic nonlinearity is still well-defined in the sense of generalized functions, which is no longer true in negative Sobolev spaces.\\
Though the present results are far from novel contributions in the Benjamin-Ono case ($a=1$, cf. \cite{Benjamin1967, Ono1975}) or the Korteweg-de Vries case ($a=2$, cf. \cite{KortewegDeVries1895}), we recall important results in the well-posedness theory of these two special cases in order to highlight peculiarities of the well-posedness theory of \eqref{eq:GeneralizedFractionalBenjaminOnoEquation}:\\
The Benjamin-Ono and Korteweg-de Vries equation were extensively studied and we mainly collect the most recent well-posedness result in the periodic case.\\
Global well-posedness of the Benjamin-Ono equation in $L^2(\mathbb{T})$ was proved by Molinet in \cite{Molinet2008}, see also the previous work \cite{Molinet2007}.\\
In these works the gauge transform, which was originally applied by Tao on the real line in \cite{Tao2004}, was transferred to the periodic case. Scaling critical regularity is $s_c = -1/2$, but the $L^2$-well-posedness result from \cite{Molinet2008} is sharp as pointed out in \cite{Molinet2009}.\\
Bourgain proved global well-posedness of the Korteweg-de Vries equation in $L^2(\mathbb{T})$ in \cite{Bourgain1993} via Picard iteration in Fourier restriction spaces.\\
The argument was refined by Kenig-Ponce-Vega in \cite{KenigPonceVega1996} to prove local well-posedness in $H^{-1/2}(\mathbb{T})$ and global well-posedness in $H^{-1/2}(\T)$ was proved by Colliander et al. in \cite{CollianderKeelStaffilaniTakaokaTao2003}.\\
It is known that the data-to-solution mapping fails to be $C^2$ below $s=-1/2$, which is thus the limit of Picard iteration. Scaling critical regularity is $s_c = -3/2$ and by non-perturbative inverse scattering arguments Kappeler-Topalov proved global well-posedness in $H^{-1}(\T)$ in \cite{KappelerTopalov2006}.\\
The argument was recently simplified by Killip-Visan in \cite{KillipVisan2018}. Sharpness of well-posedness in $H^{-1}(\T)$ was proved by Molinet in \cite{Molinet2012}.\\
For the dispersion generalized equations on the real line
\begin{equation}
\label{eq:DispersionGeneralizedBenjaminOnoRealLine}
\left\{ \begin{array}{cl}
\partial_t u + \partial_x D_x^a u &= u \partial_x u \quad (t,x) \in \R \times \R, \\
u(0) &= u_0 \in H^s(\R) \end{array} \right.
\end{equation}
global well-posedness in $L^2(\R)$ was proved by Herr et al. in \cite{HerrIonescuKenigKoch2010} adjusting the gauge transform for $1<a<2$. Carrying out this approach brought up substantial technical difficulties due to the strong dependence of the gauge on the frequencies.\\
Notably, in a previous work by Herr \cite{Herr2007} was shown that after weakening the problematic $High \times Low \rightarrow High$-interaction through introducing a low-frequency weight \eqref{eq:DispersionGeneralizedBenjaminOnoRealLine} becomes amenable to Picard iteration for $1 < a <2$ and sharp local well-posedness results were established.\\
A much simpler approach than the one from \cite{HerrIonescuKenigKoch2010} was pointed out recently by Molinet and Vento in \cite{MolinetVento2015}, where local well-posedness for $s \geq 1-a/2$ was proved as well on the real line as on the circle. In \cite{MolinetVento2015} Fourier restriction spaces are utilized in a novel iteration scheme combined with energy arguments to ameliorate the derivative loss. This very accessible approach does not rely on frequency dependent time localization, but on comprehension of the resonance function.\\
In the present work the following result is proved by short time analysis:
\begin{theorem}
\label{thm:WellposednessFractionalBenjaminOno}
For $1<a \leq 3/2$ \eqref{eq:GeneralizedFractionalBenjaminOnoEquation} is locally well-posed in $H^s(\mathbb{T})$ provided that $s>3/2-a$ and for $3/2 < a < 2$ \eqref{eq:GeneralizedFractionalBenjaminOnoEquation} is globally well-posed in $L^2(\mathbb{T})$.
\end{theorem}
\begin{remark}
Molinet pointed out in \cite{Molinet2008} that in the Benjamin-Ono case the periodic data-to-solution mapping is $C^\infty$ on hyperplanes of initial data with fixed mean. From this one might suspect that this is also true in the dispersion generalized case. However, Herr proved in \cite{Herr2008} that \eqref{eq:GeneralizedFractionalBenjaminOnoEquation} can not be solved via Picard iteration for $1 \leq a < 2$ justifying the use of short time analysis.
\end{remark}
The analysis extends and generalizes the short time analysis from \cite{Guo2012DispersionGeneralizedBenjaminOno} on the real line, which is further improved by considering modified energies. By this we mean correction terms for the frequency localized energy corresponding to normal form transformations in the spirit of the $I$-method (cf. \cite{CollianderKeelStaffilaniTakaokaTao2003}), but without symmetrization.\\
The improved symmetrized expression does not yield new information when analyzing differences of solutions because of reduced symmetry, still normal form transformations allow us to improve the energy estimates.\\
An early application of modified energies was given by Kwon in \cite{Kwon2008}, also in the context of derivative nonlinearities. In the context of short time analysis applications were given by Guo-Kwak-Kwon in \cite{GuoKwakKwon2013} and Kwak in \cite{Kwak2016}.\\
On the real line short time analysis for dispersion generalized Benjamin-Ono equations was already carried out in \cite{Guo2012DispersionGeneralizedBenjaminOno} without normal form transformation giving local well-posedness for $s \geq 2-a$, where $1 \leq a \leq 2$.\\
For previous works making use of frequency dependent time localization close to the present context see \cite{IonescuKenigTataru2008} or \cite{Molinet2012,GuoOh2018} in the periodic case.\\
The gain from modified energies is most significant for large dispersion coefficients allowing us to prove well-posedness in $L^2(\mathbb{T})$. Further, it appears as if some of the arguments can be applied in the low-dispersion case $0<a<1$. For these equations on the circle, which are also of physical interest, are currently no well-posedness results beyond the classical energy method available.\\
On the real line there is the recent work by Molinet-Pilod-Vento \cite{MolinetPilodVento2018} refining the analysis from \cite{MolinetVento2015} by normal form transformations. Since this analysis makes use of smoothing effects on the real line, which are not available on the circle, it is not clear how to extend the analysis from \cite{MolinetPilodVento2018} to the circle.\\
The local well-posedness result from Theorem \ref{thm:WellposednessFractionalBenjaminOno} for $1 < a <2$, which is globalized for $a>3/2$ due to conservation of mass on $\mathbb{T}$ is currently the best. Previously, global well-posedness for $s \geq 1-a/2$, where $1<a<2$, proved in \cite{MolinetVento2015} was the best currently available result.\\
The analysis can be transferred to the real line. On the real line, the multilinear estimates relying on linear and bilinear Strichartz estimates would be improved due to dispersive effects (thus, we prefer to analyze the more involved periodic case), however, the introduction of a modified energy would require additional care because the resonance
\begin{equation*}
\Omega(\xi_1,\xi_2,\xi_3) = \xi_1 |\xi_1|^a + \xi_2 |\xi_2|^a + \xi_3 |\xi_3|^a \quad \xi_i \in \mathbb{R}, \; \xi_1+\xi_2+\xi_3 = 0
\end{equation*}
might become arbitrary small in modulus for non-vanishing $\xi_i \in \mathbb{R}$. To avoid this we confine ourselves to initial data with vanishing mean. As this is a conserved quantity, there is no loss of generality in assuming
\begin{equation*}
\int_{\mathbb{T}} u(x) dx = 0
\end{equation*}
Regarding the frequency dependent time localization we will not work in Euclidean windows (cf. \cite{rsc2018BilinearStrichartzEstimates}), but rather base the analysis on the time localization $T=T(N)=N^{a-2}$ interpolating between Euclidean windows in the Benjamin-Ono case and the Fourier restriction norm analysis from \cite{Bourgain1993} for $a=2$, where frequency dependent time localization is no longer required. For the large data theory it turns out to be convenient to consider the slightly shorter times $N^{a-2-\delta}$ giving an additional factor of $T^\theta$ in the nonlinear estimates (cf. Lemma \ref{lem:tradingModulationRegularity}).\\
The following set of estimates will be established for the proof of Theorem \ref{thm:WellposednessFractionalBenjaminOno} for a smooth solution $u$ to \eqref{eq:GeneralizedFractionalBenjaminOnoEquation} with vanishing mean. For $1 < a < 2$, $T \in (0,1]$, $M \in 2^{\mathbb{N}_0}$ and $s^\prime \geq s \geq \max(3/2-a,0)$ there are $\delta(a,s) > 0$, $c(a,s) > 0$, $d(a,s) > 0$ and $\theta(a,s) > 0$ such that
\begin{equation*}
\left\{ \begin{array}{cl}
\Vert u \Vert_{F^{s^\prime,\delta}_a(T)} &\lesssim \Vert u \Vert_{E^{s^\prime}(T)} + \Vert u \partial_x u \Vert_{N^{s^\prime,\delta}_a(T)} \\
\Vert u \partial_x u \Vert_{N^{s^\prime,\delta}_a(T)} &\lesssim T^\theta \Vert u \Vert_{F^{s^\prime,\delta}_a(T)} \Vert u \Vert_{F^{s,\delta}_a(T)} \\
\Vert u \Vert^2_{E^{s^\prime}(T)} &\lesssim \Vert u(0) \Vert^2_{H^{s^\prime}(T)} + M^{c} T \Vert u \Vert^2_{F^{s^\prime,\delta}_a(T)} \Vert u \Vert_{F^{s,\delta}_a(T)} \\
&+ M^{-d} \Vert u \Vert_{F^{s^\prime,\delta}_a(T)}^2 \Vert u \Vert_{F^{s,\delta}_a(T)} + T^\theta \Vert u \Vert^2_{F^{s^\prime,\delta}_a(T)} \Vert u \Vert^2_{F^{s,\delta}_a(T)}
\end{array} \right.
\end{equation*}
By the usual bootstrap arguments (cf. \cite{IonescuKenigTataru2008,GuoOh2018}) the above display gives a priori estimates.\\
For differences of solutions $v=u_1-u_2$, where $u_i$ denote smooth solutions to \eqref{eq:GeneralizedFractionalBenjaminOnoEquation} with vanishing mean, we have the following set of estimates for $s > 3/2 -a$ in case $1<a \leq 3/2$ and $s=0$ in case $3/2 < a <2$ and the remaining parameters like in the previous display:
\begin{equation*}
\left\{\begin{array}{cl}
\Vert v \Vert_{F^{-1/2,\delta}_a(T)} &\lesssim \Vert v \Vert_{E^{-1/2}(T)} + \Vert \partial_x (v (u_1+u_2)) \Vert_{N^{-1/2,\delta}_a(T)} \\
\Vert \partial_x((u_1+u_2) v) \Vert_{N_a^{-1/2,\delta}(T)} &\lesssim T^\theta \Vert v \Vert_{F^{-1/2,\delta}_a(T)} ( \Vert u_1 \Vert_{F^{s,\delta}_a(T)} + \Vert u_2 \Vert_{F^{s,\delta}_a(T)} ) \\
\Vert v \Vert^2_{E^{-1/2}(T)} &\lesssim \Vert v(0) \Vert^2_{H^{-1/2}} \\
&+ M^{c} T \Vert v \Vert^2_{F^{-1/2,\delta}_a(T)} ( \Vert u_1 \Vert_{F^{s,\delta}_a(T)} + \Vert u_2 \Vert_{F^{s,\delta}_a(T)}) \\
&+ M^{-d} \Vert v \Vert_{F^{-1/2,\delta}_a(T)}^2 ( \Vert u_1 \Vert_{F^{s,\delta}_a(T)} + \Vert u_2 \Vert_{F^{s,\delta}_a(T)}) \\
&+ T^\theta \Vert v \Vert^2_{F^{-1/2,\delta}_a(T)} ( \Vert u_1 \Vert^2_{F^{s,\delta}_a(T)} + \Vert u_2 \Vert_{F^{s,\delta}_a(T)}^2 )
\end{array} \right.
\end{equation*}
which yields Lipschitz-continuity in $H^{-1/2}$ for initial data in $H^s$.\\
The related set of estimates with parameters like in the previous display
\begin{equation*}
\left\{\begin{array}{cl}
\Vert v \Vert_{F_a^{s,\delta}(T)} &\lesssim \Vert v \Vert_{E^s(T)} + \Vert \partial_x (v(u_1 + u_2)) \Vert_{N_a^{s,\delta}(T)} \\
\Vert \partial_x (v(u_1 + u_2)) \Vert_{N_a^{s,\delta}(T)} &\lesssim T^\theta \Vert v \Vert_{F_a^{s,\delta}(T)} \left( \Vert u_1 \Vert_{F_a^{s,\delta}(T)} + \Vert u_2 \Vert_{F_a^{s,\delta}(T)} \right) \\
\Vert v \Vert^2_{E^{s}(T)} &\lesssim \Vert v(0) \Vert^2_{H^s} \\
&+ M^c T \Vert v \Vert^2_{F^{s,\delta}_a(T)} (\Vert u_2 \Vert_{F^{s,\delta}_a(T)} + \Vert v \Vert_{F_a^{s,\delta}(T)}) \\
&+ M^{-d} \Vert v \Vert^2_{F_a^{s,\delta}(T)} ( \Vert u_2 \Vert_{F_a^{s,\delta}(T)} + \Vert v \Vert_{F^{s,\delta}_a(T)}) \\
&+ T^\theta (\Vert v \Vert^2_{F_a^{s,\delta}(T)} ( \Vert u_2 \Vert^2_{F_a^{s,\delta}(T)} + \Vert v \Vert_{F_a^{s,\delta}(T)}^2 ) \\
&\quad + \Vert v \Vert_{F_a^{-1/2,\delta}(T)} \Vert v \Vert_{F_a^{s,\delta}(T)} \Vert u_2 \Vert_{F_a^{r,\delta}(T)} \Vert u_2 \Vert_{F_a^{s,\delta}(T)} ),
\end{array} \right.
\end{equation*}
where $r = (2-a)+s$, yields continuous dependence by a variant of the Bona-Smith approximation (cf. \cite{IonescuKenigTataru2008,GuoOh2018,BonaSmith1975}). The conclusion of Theorem \ref{thm:WellposednessFractionalBenjaminOno} from the above set of estimates is standard and thus omitted.\\
The article is structured as follows: In Section \ref{section:Notation} notation and function spaces are introduced, in Section \ref{section:LinearBilinearEstimates} linear and bilinear estimates for frequency localized functions are discussed. These estimates are applied in Section \ref{section:ShorttimeBilinearEstimates} to derive an estimate for the nonlinearity in short time function spaces and in Section \ref{section:EnergyEstimates} the energy norm is propagated in short time function spaces.
\section{Notation and Function spaces}
\label{section:Notation}
Purpose of this section is to fix notation and introduce function spaces. For the proofs of the basic function space properties, which hold true independent of the domain and dispersion relation, we will refer to the literature.\\
The Fourier transform of a $2 \pi$-periodic $L^1$-function $f: \mathbb{T} \rightarrow \mathbb{C}$ takes on values in $\mathbb{Z}$ and is defined by
\begin{equation}
\label{eq:FourierTransformPeriodicFunction}
\hat{f}(\xi) = \int_{ \mathbb{T} } f(x) e^{- ix \xi} dx \quad (\xi \in \mathbb{Z})
\end{equation}
Below we will occasionally write $ (d\xi)_1$ for the counting measure on $\Z$ to emphasize similarity to the real line case
\begin{equation}
\label{eq:normalizedCountingMeasure}
\int a(\xi) (d\xi)_1 := \sum_{\xi \in \Z} a(\xi)
\end{equation}
$d \Gamma_n(\xi_1,\ldots,\xi_n)$ denotes the measure on the hypersurface $\Gamma_n = \{(\xi_1,\ldots,\xi_n) \in \mathbb{Z}^n \; | \; \xi_1 + \ldots + \xi_n = 0 \}$:
\begin{equation*}
\int_{\Gamma_n} f(\xi_1,\ldots,\xi_{n}) d\Gamma_n(\xi_1,\ldots,\xi_n) := \sum_{\xi_1,\ldots,\xi_{n-1} \in \mathbb{Z}} f(\xi_1,\ldots,\xi_{n-1},-\xi_1-\ldots-\xi_{n-1})
\end{equation*}
The Fourier inversion formula is given by
\begin{equation}
\label{eq:conventionFourierInversion}
f(x) = \frac{1}{2 \pi} \int \hat{f}(\xi) e^{ix \xi} (d\xi)_{1}
\end{equation}
We define the Sobolev space $H^s(\T)$ ($H^s$ for brevity) with norm
\begin{equation}
\label{eq:definitionSobolevSpace}
\Vert f \Vert_{H^s(\T)} = \Vert \hat{f}(\xi) \langle \xi \rangle^{s} \Vert_{L^2_{(d\xi)_1}}
\end{equation}
and $H^\infty(\T) = \bigcap_s H^s(\T)$.\\
For a $2 \pi$-space-periodic function $f(x,t)$ with time variable $t \in \R$, we define the space-time Fourier transform
\begin{equation}
\label{eq:PeriodicSpaceTimeFourierTransform}
\tilde{v}(\tau,\xi) = (\mathcal{F}_{t,x} v)(\tau,\xi) = \int_{\R} dt \int_{\T} dx e^{-ix \xi} e^{-it \tau} v(t,x) \quad (\xi \in \Z, \; \tau \in \R)
\end{equation}
The periodic space-time Fourier transform is inverted by
\begin{equation}
\label{eq:InversionPeriodicSpaceTimeFourierTransform}
v(t,x) = \frac{1}{(2 \pi)^{2}} \int \int e^{i x \xi} e^{i t \tau} \tilde{v}(\tau,\xi) (d\xi)_1 d\tau
\end{equation}
We define Littlewood-Paley projectors on the circle: Let $\rho: \mathbb{R} \rightarrow \R_{\geq 0}$ be a smooth and radially decreasing function with
\begin{equation*}
\rho(\xi) \equiv 1, \quad |\xi| \leq 1 \text{ and supp} \; \rho \subseteq B(0,2)
\end{equation*}
For $k \in \mathbb{N}$ define
\begin{equation*}
\chi_k(\xi) = \rho(2^{-k} \xi) - \rho(2^{1-k} \xi), \quad \quad \text{supp} \chi_k \subseteq B(0,2^{k+1}) \backslash B(0,2^{k-1})
\end{equation*}
and for $f \in L^2(\mathbb{T})$ the $k$th Littlewood-Paley projector is defined by
\begin{equation*}
(P_k f) \widehat (\xi) = \chi_k(\xi) \hat{f}(\xi).
\end{equation*}
The zero-frequency will be considered separately:
\begin{equation}
\label{eq:LowFrequencyProjector}
P_{0} f \widehat (\xi) = 1_{(-1/2,1/2)}(\xi) \hat{f}(\xi)
\end{equation}
The definition of the function spaces requires a partition in the modulation which we will denote differently from the partition of the spatial frequencies.\\
Let $\eta_0: \R \rightarrow [0,1]$ denote an even, smooth function $\text{supp} \, (\eta_0) \subseteq [-5/4,5/4]$. For $k \in \mathbb{N}$ we set
\begin{equation*}
\eta_k(\tau) = \eta_0(\tau/2^k) - \eta_0(\tau/2^{k-1}) \quad \quad 
\end{equation*}
We write $\eta_{\leq m} = \sum_{j=0}^m \eta_j$ for $m \in \mathbb{N}$.\\
The dispersion relation will be denoted by
\begin{equation*}
\varphi_a(\xi) = \xi | \xi|^a
\end{equation*}
The regions in Fourier space localized at frequency and modulation are denoted by
\begin{equation*}
D^a_{k_i,j_i} = \{ (\xi, \tau) \in \mathbb{Z} \times \mathbb{R} \, | \, |\xi| \sim 2^{k_i}, \; | \tau - \varphi_a(\xi)| \sim 2^{j_i} \}
\end{equation*}
with the obvious modification for the variant $D^a_{k_i,\leq j_i}$.\\
Next, we define an $X^{s,b}$-type space for the Fourier transform of frequency-localized periodic functions:
\begin{equation*}
\begin{split}
\label{eq:XkDefinition}
&X_{a,k} = \{ f:  \mathbb{R} \times \mathbb{Z} \rightarrow \mathbb{C} \; | \\
 &\mathrm{supp}(f) \subseteq \R \times {I_k}, \Vert f \Vert_{X_{a,k}} = \sum_{j=0}^\infty 2^{j/2} \Vert \eta_j(\tau - \varphi_a(\xi)) f(\tau,\xi) \Vert_{\ell^2_{\xi}  L^2_{\tau}} < \infty \} .
\end{split}
\end{equation*}
Partitioning the modulation variable through a sum over $\eta_j$ yields the estimate
\begin{equation}
\label{eq:XkEstimateI}
\Vert \int_{\mathbb{R}} | f_k(\tau^\prime,\xi) | d\tau^\prime \Vert_{\ell^2_{\xi}} \lesssim \Vert f_k \Vert_{X_{a,k}}. 
\end{equation}
Also, we record the estimate
\begin{equation}
\label{eq:XkEstimateII}
\begin{split}
&\sum_{j=l+1}^{\infty} 2^{j/2} \Vert \eta_j(\tau - \varphi_a(\xi)) \cdot \int_{\mathbb{R}} | f_k(\tau^\prime,\xi) | \cdot 2^{-l} (1+ 2^{-l}|\tau - \tau^\prime |)^{-4} d\tau^\prime \Vert_{L^2_{(d\xi)_1} L^2_\tau} \\
&+ 2^{l/2} \Vert \eta_{\leq l} (\tau - \varphi_a(\xi)) \cdot \int_{\mathbb{R}} | f_k(\tau^\prime,\xi) | \cdot 2^{-l} (1+ 2^{-l}|\tau - \tau^\prime |)^{-4} d\tau^\prime \Vert_{L^2_{(d\xi)_1} L^2_\tau }\\
&\lesssim \Vert f_k \Vert_{X_{a,k}}
\end{split}
\end{equation}
which is an instance of \cite[Equation~(3.5)]{GuoOh2018}.\\
In particular, we find for a Schwartz-function $\gamma$ for $k, l \in \mathbb{N}, t_0 \in \mathbb{R}, f_k \in X_{a,k}$ the estimate
\begin{equation}
\label{eq:XkEstimateIII}
\Vert \mathcal{F}[\gamma(2^l(t-t_0)) \cdot \mathcal{F}^{-1}(f_k)] \Vert_{X_{a,k}} \lesssim_{\gamma} \Vert f_k \Vert_{X_{a,k}}
\end{equation}
We define a dyadically localized energy space
\begin{equation*}
E_k = \{ f \in L^2 | P_k f = f \}
\end{equation*}
and set
\begin{equation*}
C_0(\R,E_k) = \{ u_k \in C(\R,E_k) \, | \, \text{supp}(u_k) \subseteq [-4,4] \times \R \}
\end{equation*}
We define the short time $X^{s,b}$-space $F_{a,k}$ for a frequency $2^k$. For $1 \leq a \leq 2$ we localize time on a scale of $2^{(a-2-\delta)k}$, where $\delta \geq 0$:
\begin{equation}
F^{\delta}_{a,k} = \{ u_k \in C_0(\R,E_k) | \Vert u_k \Vert_{F^{\delta}_{a,k}} = \sup_{t_k \in \R} \Vert \mathcal{F}[u_k \eta_0(2^{(2-a+\delta)k} (t-t_k))] \Vert_{X_{a,k}} < \infty \}
\end{equation}
Based on the observation that for $a=1$ $T=T(N)=N^{-1}$ is a natural localization in time (cf. \cite{rsc2018BilinearStrichartzEstimates}) and that for $a=2$ we do not need localization in time anymore to overcome the derivative loss due to sufficient dispersive effects we choose as inbetween localization in time $T=T(N)=N^{a-2-\delta}$. It turns out that for some limiting cases $\delta>0$ will be useful.\\
Correspondingly, we define the space in which the nonlinearity will be estimated as
\begin{equation*}
\begin{split}
&N^{\delta}_{a,k} = \{ u_k \in C_0(\R,E_k) | \\
&\Vert u_k \Vert_{N^{\delta}_{k,a}} = \sup_{t_k \in \R} \Vert (\tau - \omega(\xi) + i2^{(2-a+\delta)k})^{-1} \mathcal{F}[u_k \eta(2^{(2-a+\delta)k}(t-t_k))] \Vert_{X_{k,a}} < \infty \}
\end{split}
\end{equation*}
We localize the spaces in time for $T \in (0,1]$ as usual:
\begin{equation*}
F^{\delta}_{a,k}(T) = \{ u_k \in C([-T,T],E_k) | \Vert u_k \Vert_{F^{\delta}_{a,k}(T)} = \inf_{\tilde{u}_k = u_k \text{ in } [-T,T]} \Vert \tilde{u}_k \Vert_{F^{\delta}_{a,k}} < \infty \}
\end{equation*}
and
\begin{equation*}
N^{\delta}_{a,k}(T) = \{ u_k \in C([-T,T],E_k) | \Vert u_k \Vert_{N^{\delta}_{a,k}(T)} = \inf_{\tilde{u}_k = u_k \text{ in } [-T,T]} \Vert \tilde{u}_k \Vert_{N^{\delta}_{a,k}} < \infty \}
\end{equation*}
The spaces $E^s$, $E^s(T)$, $F_a^{s,\delta}(T)$ and $N_a^{s,\delta}(T)$ are composed via Littlewood-Paley decomposition:
\begin{equation*}
E^s = \{ f: \R \rightarrow \C | \Vert \phi \Vert^2_{E^s} = \sum_{k \geq 0} 2^{2ks} \Vert P_k f \Vert^2_{L^2} < \infty \}
\end{equation*}
and for the solution we define
\begin{equation*}
\begin{split}
E^s(T) = \{ u \in C([-T,T],H^\infty) | \Vert u \Vert^2_{E^s(T)} &= \Vert P_0 u(0) \Vert_{L^2}^2 \\
&+ \sum_{k \geq 1} \sup_{t_k \in [-T,T]} 2^{2ks} \Vert P_k u(t_k) \Vert^2_{L^2} < \infty \}.
\end{split}
\end{equation*}
We define the short time $X^{s,b}$-space for the solution
\begin{equation*}
F^{s,\delta}_a(T) = \{ u \in C([-T,T],H^\infty) | \Vert u \Vert^2_{F_a^{s,\delta}(T)} = \sum_{k \geq 0} 2^{2ks} \Vert P_k u \Vert^2_{F^{\delta}_{k,a}(T)} < \infty \},
\end{equation*}
and for the nonlinearity we define
\begin{equation*}
N^{s,\delta}_a(T) = \{ u \in C([-T,T],H^\infty) | \Vert u \Vert_{N_a^{s,\delta}(T)} = \sum_{k \geq 0} 2^{2ks} \Vert P_k u \Vert_{F^{\delta}_{a,k}(T)} < \infty \}.
\end{equation*}
We will also make use of $k$-acceptable time multiplication factors (cf. \cite{IonescuKenigTataru2008}): For $k \in \mathbb{N}_0$ we set
\begin{equation*}
S^{\delta}_{a,k} = \{ m_k \in C^{\infty}(\mathbb{R},\mathbb{R}) : \; \Vert m_k \Vert_{S_k} = \sum_{j=0}^{10} 2^{-j (2-a+\delta) k} \Vert \partial^j m_k \Vert_{L^{\infty}} < \infty \}.
\end{equation*}
The generic example is given by time localization on a scale of $2^{-(2-a+\delta) k}$, i.e., $\eta_0(2^{(2-a+\delta) k} \cdot)$.\\ 
The estimates (cf. \cite[Eq.~(2.21),~p.~273]{IonescuKenigTataru2008})
\begin{equation}
\label{eq:timeLocalizationShorttimeNorms}
\begin{split}
\left\{\begin{array}{cl}
\Vert \sum_{k \geq 0} m_k(t) P_k(u) \Vert_{F_{a,k}^{\delta}(T)} \lesssim (\sup_{k \geq 0} \Vert m_k \Vert_{S^{\delta}_{a,k}}) \cdot \Vert u \Vert_{F_{a,k}^{\delta}(T)}, \\
\Vert \sum_{k \geq 0} m_k(t) P_k(u) \Vert_{N_{a,k}^{\delta}(T)} \lesssim (\sup_{k \geq 0} \Vert m_k \Vert_{S^{\delta}_{a,k}}) \cdot \Vert u \Vert_{N_{a,k}^{\delta}(T)},
\end{array} \right.
\end{split}
\end{equation}
follow from integration by parts. From \eqref{eq:timeLocalizationShorttimeNorms} follows that we can assume $F^{\delta}_{a,k}(T)$ functions to be supported in time on an interval $[-T-2^{-\alpha k-10},T+2^{-\alpha k-10}]$.\\
We record basic properties of the shorttime $X^{s,b}$-spaces introduced above. The next lemma establishes the embedding $F_{a}^{s,\delta}(T) \hookrightarrow C([0,T],H^s)$.
\begin{lemma}
\label{lem:FsEmbedding}
\begin{enumerate}
\item[(i)]
We find the estimate
\begin{equation*}
\Vert u \Vert_{L_t^\infty L_x^2} \lesssim \Vert u \Vert_{F^{\delta}_{a,k}}
\end{equation*}
to hold for any $u \in F^{\delta}_{a,k}$.
\item[(ii)]
Suppose that $s \in \R$, $T>0$ and $u \in F_{a}^{s,\delta}(T)$. Then, we find the estimate
\begin{equation*}
\Vert u \Vert_{C([0,T],H^s)} \lesssim \Vert u \Vert_{F_{a}^{s,\delta}(T)}
\end{equation*}
to hold.
\end{enumerate}
\end{lemma}
\begin{proof}
For a proof see \cite[Lemma~3.1.,~p.~274]{IonescuKenigTataru2008} in Euclidean space and\\
\cite[Lemma~3.2,~3.3]{GuoOh2018} in the periodic case.
\end{proof}
We state the energy estimate for the above short time $X^{s,b}$-spaces. The proof which was carried out on the real line in \cite[Proposition~3.2.,~p.~274]{IonescuKenigTataru2008} and in the periodic case in \cite[Proposition~4.1.]{GuoOh2018} is omitted.
\begin{proposition}
\label{prop:linearShorttimeEnergyEstimate}
Let $T \in (0,1]$, $1<a<2$ and $u, v \in C([-T,T],H^{\infty})$ satisfy the equation
\begin{equation*}
\partial_t u + \varphi_a(\nabla/i) u =v \; \mathrm{ in } \; \mathbb{T} \times (-T,T).
\end{equation*}
Then, we find the following estimate to hold for any $s \in \mathbb{R}$:
\begin{equation*}
\Vert u \Vert_{F_{a}^{s,\delta}(T)} \lesssim \Vert u \Vert_{E^{s,\delta}(T)} + \Vert v \Vert_{N_{a}^{s,\delta}(T)}
\end{equation*}
\end{proposition}
For the large data theory we have to define the following generalizations in terms of regularity in the modulation variable to the $X_{a,k}$-spaces:
\begin{equation*}
\begin{split}
&X^b_{a,k} = \{ f: \mathbb{R} \times \mathbb{Z} \rightarrow \mathbb{C} \; | \\
 &\mathrm{supp}(f) \subseteq \mathbb{R} \times {I_k}, \Vert f \Vert_{X^b_{a,k}} = \sum_{j=0}^\infty 2^{bj} \Vert \eta_j(\tau - \varphi_a(\xi)) f(\tau,\xi) \Vert_{\ell^2_{\xi} L^2_{\tau}} < \infty \},
\end{split}
\end{equation*}
where $b \in \R$. The short time spaces $F^{b,\delta}_{a,k}$, $F^{b,s,\delta}_{a}(T)$ and $N^{b,\delta}_{a,k}$, $N^{b,s,\delta}_{a}(T)$ are defined following along the above lines with $X_{a,k}$ replaced by $X^b_{a,k}$.\\
Indeed, in a similar spirit to the treatment of $X^{s,b}_T$-spaces we can trade regularity in the modulation variable for a small power of $T$:
\begin{lemma}{\cite[Lemma~3.4]{GuoOh2018}}
\label{lem:tradingModulationRegularity}
Let $T>0$, $1<a<2$, $\delta  \geq 0$ and $b<1/2$. Then, we find the following estimate to hold:
\begin{equation*}
\Vert P_k u \Vert_{F_{a,k}^{b,\delta}} \lesssim T^{(1/2-b)-} \Vert P_k u \Vert_{F_{a,k}^\delta}
\end{equation*}
for any function $u$ with temporal support in $[-T,T]$.
\end{lemma}
Below we will have to consider the action of sharp time cutoffs in the $X_k$-spaces. Recall from the usual $X^{s,b}$-space-theory that multiplication with a sharp cutoff in time is not bounded. However, we find the following estimate to hold:
\begin{lemma}{\cite[Lemma~3.5]{GuoOh2018}}
\label{lem:sharpTimeCutoffAlmostBounded}
Let $k \in \Z$. Then, for any interval $I=[t_1,t_2] \subseteq \R$, we find the following estimate to hold:
\begin{equation*}
\sup_{j \geq 0} 2^{j/2} \Vert \eta_j(\tau-\varphi_a(\xi)) \mathcal{F}_{t,x}[1_{I}(t) P_k u] \Vert_{L_\tau^2 \ell^2_\xi} \lesssim \Vert \mathcal{F}_{t,x} (P_k u) \Vert_{X_{a,k}}
\end{equation*}
with implicit constant independent of $k$ and $I$.
\end{lemma}
\section{Linear and bilinear estimates}
\label{section:LinearBilinearEstimates}
In the following we derive $L^2$-bilinear convolution estimates for space-time functions localized in frequency and modulation. Consider $k_i,j_i$, $i=1,2,3$ and $f_{k_i,j_i} \in L^2_{\geq 0}(\mathbb{Z} \times \mathbb{R})$, $\text{supp} (f_{k_i,j_i}) \subseteq D^a_{k_i,\leq j_i}$. Aim is to prove estimates
\begin{equation}
\label{eq:L2BilinearEstimateBenjaminOno}
\int \int f_{k_1,j_1}(\xi_1,\tau_1) f_{k_2,j_2}(\xi_2,\tau_2) f_{k_3,j_3}(\xi_3,\tau_3) d\Gamma_3(\xi) d\Gamma_3(\tau) \lesssim \alpha(\underline{k},\underline{j}) \prod_{i=1}^3 \Vert f_{k_i,j_i} \Vert_{L^2}
\end{equation}
The following $L^4_{t,x}$-Strichartz estimate is independent of the separation of the frequencies. The proof generalizes the $a=2$-case given in \cite[Lemma~3.3.,~p.~1906]{Molinet2012}.
\begin{lemma}
\label{lem:L4StrichartzEstimateFractionalBenjaminOno}
Let $1 \leq a \leq 2$, $f_{k_i,j_i} \in L^2_{\geq 0}(\Z \times \R)$, $\text{supp} f_{k_i,j_i} \subseteq D^a_{k_i,\leq j_i}$, $i=1,2$. Then, we find the following estimate to hold:
\begin{equation}
\label{eq:L4StrichartzEstimateFractionalDispersion}
\Vert f_{k_1,j_1} * f_{k_2,j_2} \Vert_{L^2_{\tau,\xi}} \lesssim 2^{j_{\min}/2} 2^{j_{\max}/(2(a+1))} \Vert f_{k_1,j_1} \Vert_2 \Vert f_{k_2,j_2} \Vert_2
\end{equation}
\end{lemma}
\begin{proof}
By the reflection lemma (\cite[Corollary~3.8.]{Tao2001})
\begin{equation*}
\Vert u v \Vert_2 = \Vert u \overline{v} \Vert_2
\end{equation*}
we can suppose that $\text{supp}_\xi f_{k_i,j_i} \subseteq \mathbb{Z}$ for $i=1,2$.\\
An application of Cauchy-Schwarz gives
\begin{equation*}
\begin{split}
&\int d\tau \int (d\xi)_1 \left| \int d\tau_1 \int (d\xi_1)_1 f_{k_1,j_1}(\xi_1,\tau_1) f_{k_2,j_2}(\xi-\xi_1,\tau-\tau_1) \right|^2 \\
&\lesssim \sup_{\tau,\xi} \alpha(\tau,\xi) \Vert f_{k_1,j_1} \Vert_2^2 \Vert f_{k_2,j_2} \Vert_2^2
\end{split}
\end{equation*}
where 
\begin{equation*}
\begin{split}
\alpha(\tau,\xi) &\lesssim \text{mes}( \{ (\tau_1,\xi_1) \in \mathbb{R} \times \mathbb{Z}_{\geq 0} | \xi- \xi_1 \in \mathbb{Z}_{\geq 0}, \langle \tau_1 -\varphi_a(\xi_1) \rangle \lesssim 2^{j_1} \\
&\text{ and } \langle \tau - \tau_1 - \varphi_a(\xi-\xi_1) \rangle \lesssim 2^{j_2} \}) \lesssim 2^{j_{\min}} \# A(\tau,\xi)
\end{split}
\end{equation*}
with
\begin{equation*}
A(\tau,\xi) = \{ \xi_1 \geq 0 | \xi-\xi_1 \geq 0 \text{ and } \langle \tau - \varphi_a(\xi_1) - \varphi_a(\xi-\xi_1) \rangle \lesssim 2^{j_{\max}} \}
\end{equation*}
In the region $2^{j_{\max}} \leq \xi^{a+1}$, notice that
\begin{equation*}
\# A(\tau,\xi) \lesssim \left( \frac{2^{j_{\max}}}{\xi^{a-1}} \right)^{1/2} +1 \lesssim 2^{j_{\max}/(a+1)}
\end{equation*}
In the region $0 \leq \xi^{a+1} \leq 2^{j_{\max}}$, use that $0 \leq \xi_1 \leq \xi$ to obtain that
\begin{equation*}
\# A(\tau,\xi) \lesssim \# \{ \xi_1 | 0 \leq \xi_1^{a+1} \leq 2^{j_{\max}} \} \lesssim 2^{j_{\max}/(a+1)}
\end{equation*}
\eqref{eq:L4StrichartzEstimateFractionalDispersion} follows from the above two displays.
\end{proof}
The following $L^6_{t,x}$-estimate is a consequence of \cite[Proposition~1.1]{rsc2019StrichartzEstimatesDecoupling} and the transfer principle (cf. \cite[Lemma~2.9,~p.~100]{Tao2006})):
\begin{lemma}
\label{lem:L6StrichartzEstimateFractionalDispersion}
Let $1<a<2$, $f_{k,j} \in L^2_{\geq 0}(\mathbb{Z} \times \mathbb{R})$ with $\text{supp} (f_{k,j}) \subseteq D^a_{k,\leq j}$. Then, we find the following estimate to hold for any $\varepsilon > 0$:
\begin{equation*}
\Vert \mathcal{F}_{t,x}^{-1}[f_{k,j}] \Vert_{L^6_{t,x}} \lesssim 2^{\varepsilon k} 2^{j/2} \Vert f_{k,j} \Vert_{L^2_{\tau} \ell^2_{\xi}}
\end{equation*}
\end{lemma}
Next, we consider multilinear refinements:
\begin{lemma}
\label{lem:L2BilinearEstimateHighLowHighInteractionBenjaminOno}
Let $|k_1-k_3| \leq 5, k_2 \leq k_1-10$. Then, we find \eqref{eq:L2BilinearEstimateBenjaminOno} to hold with
\begin{equation*}
\begin{split}
\alpha(\underline{k},\underline{j}) = \min( &(1+2^{j_3-ak_1})^{1/2} 2^{j_2/2}, (1+2^{j_2-ak_1})^{1/2} 2^{j_1/2}, \\ 
&(1+2^{j_3-(a-1)k_1-k_2})^{1/2} 2^{j_1/2} )
\end{split}
\end{equation*}
\end{lemma}
\begin{proof}
We perform a change of variables $f_{k_i,j_i}^{\#}(\xi,\tau) = f_{k_i,j_i}(\xi,\tau+\varphi_a(\xi))$ so that $\Vert f_{k_i,j_i}^{\#} \Vert_2 = \Vert f_{k_i,j_i} \Vert_2$ and
$\text{supp} (f_{k_i,j_i}^{\#}) \subseteq \{ (\xi_i,\tau_i) \in \mathbb{Z} \times \mathbb{R} \; | \; |\xi_i| \sim 2^{k_i}, \, |\tau_i| \lesssim 2^{j_i} \}$\footnote{Actually, in the following computations we freely interchange $f$ with $\tilde{f}(\xi,\tau) = f(-\xi,-\tau)$ as $\Vert \tilde{f} \Vert_2 = \Vert f \Vert_2$.}.\\
The resonance function 
\begin{equation}
\label{eq:dispersionGeneralizedResonanceFunctionBenjaminOno}
\Omega^a(\xi_1,\xi_2) = (\xi_1+\xi_2)|\xi_1+\xi_2|^a - \xi_1 |\xi_1|^a - \xi_2 |\xi_2|^a
\end{equation}
will come into play quantifying the effective support of the involved functions. Record
\begin{equation}
\label{eq:GroupVelocitiesHighLowBenjaminOno}
\begin{split}
\left| \frac{\partial \Omega^a}{\partial \xi_1} \right| &= ||\xi_1+\xi_2|^a - |\xi_1|^a| \sim |\xi_1|^{a-1} |\xi_2| \\
\left| \frac{\partial \Omega^a}{\partial \xi_2} \right| &= ||\xi_1+\xi_2|^a - |\xi_2|^a| \sim |\xi_1+\xi_2|^a
\end{split}
\end{equation}
We prove the first estimate. An application of Cauchy-Schwarz inequality in $\xi_2$ yields
\begin{equation*}
\begin{split}
&\int \int f_{k_1,j_1}(\xi_1,\tau_1) f_{k_2,j_2}(\xi_2,\tau_2) f_{k_3,j_3}(\xi_3,\tau_3) d\Gamma_3(\xi) d\Gamma_3(\tau) \\
&= \int (d\xi_1)_1 \int d\tau_1 f_{k_1,j_1}^{\#}(\xi_1,\tau_1) \int d\tau_2 (1+2^{j_3-ak_1})^{1/2} \\
&\left( \int (d\xi_2)_1 |f_{k_2,j_2}^{\#}(\xi_2,\tau_2)|^2 |f_{k_3,j_3}^{\#}(\xi_1+\xi_2,\tau_1+\tau_2+\Omega)|^2 \right)^{1/2}
\end{split}
\end{equation*}
Further applications of Cauchy-Schwarz in $\tau_1,\xi_1$ and $\tau_2$ yield
\begin{equation*}
\begin{split}
&\lesssim \int (d\xi_1)_1 \int d\tau_2 (1+2^{j_3-ak_1})^{1/2} \left( \int d\tau_1 |f_{k_1,j_1}^{\#}(\xi_1,\tau_1)|^2 \right)^{1/2} \\
&\times \left( \int (d\xi_2)_1 |f_{k_2,j_2}^{\#}(\xi_2,\tau_2)|^2 \int d\tau_1 |f_{k_3,j_3}^{\#}(\xi_1+\xi_2,\tau_1+\tau_2+\Omega^a)|^2 \right)^{1/2} \\
&\lesssim (1+2^{j_3-ak_1})^{1/2} \int d\tau_2 \Vert f_{k_1,j_1}^{\#}\Vert_2 \left( \int (d\xi_2)_1 |f_{k_2,j_2}^{\#}(\xi_2,\tau_2)|^2 \right)^{1/2} \Vert f^{\#}_{k_3,j_3} \Vert_{L^2} \\
&\lesssim 2^{j_2/2} (1+2^{j_3-ak_1})^{1/2} \prod_{i=1}^3 \Vert f^{\#}_{k_i,j_i} \Vert_{L^2}
\end{split}
\end{equation*}
This yields the first bound.\\
For the second claim carry out the same computation after rearranging
\begin{equation*}
\begin{split}
&\int \int f_{k_3,j_3}(\xi_3,\tau_3) f_{k_1,j_1}(\xi_1,\tau_1) f_{k_2,j_2}(\xi_2,\tau_2) d\Gamma_3(\xi) d\Gamma_3(\tau) \\
&= \int (d\xi_3)_1 \int d\tau_3 f_{k_3,j_3}^{\#}(\xi_3,\tau_3) \\
&\int (d\xi_1)_1 \int d\tau_1 f_{k_1,j_1}^{\#}(\xi_1,\tau_1) f_{k_2,j_2}^{\#}(\xi_1+\xi_3,\tau_1+\tau_3+\Omega^a(\xi_1,\xi_3))
\end{split}
\end{equation*}
Note that
\begin{equation*}
\left| \frac{\partial \Omega^a(\xi_1,\xi_3)}{\partial \xi_1} \right| \sim | |\xi_1+\xi_3|^a - |\xi_1|^a| \sim |\xi_1|^a
\end{equation*}
Firstly, apply Cauchy-Schwarz in $\xi_1$ to find
\begin{equation*}
\begin{split}
&\lesssim \int (d\xi_3)_1 \int d\tau_3 f_{k_3,j_3}^{\#}(\xi_3,\tau_3) \int d\tau_1 (1+2^{j_2-ak_1})^{1/2} \\
&\times \left( \int (d\xi_1)_1 |f_{k_1,j_1}^{\#}(\xi_1,\tau_1)|^2 |f_{k_2,j_2}^{\#}(\xi_1+\xi_3,\tau_1+\tau_3+\Omega^a)|^2 \right)^{1/2} 
\end{split}
\end{equation*}
and next, apply Cauchy-Schwarz in $\tau_3$, $\xi_3$ and at last $\tau_1$ to find the bound
\begin{equation*}
\lesssim 2^{j_1/2} (1+2^{j_2-ak_1})^{1/2} \prod_{i=1}^3 \Vert f_{k_i,j_i} \Vert_2
\end{equation*}
The third bound will be established by the same argument. The difference of the group velocity is less favourable though, leading to inferior estimates: An application of the Cauchy-Schwarz inequality in $\xi_1$ yields
\begin{equation*}
\begin{split}
&\int d\tau_2 \int (d\xi_2)_1 f_{k_2,j_2}^{\#} (\xi_2,\tau_2) \int (d\xi_1)_1 \int d\tau_1 f_{k_1,j_1}^{\#}(\xi_1,\tau_1) f_{k_3,j_3}^{\#}(\xi_1+\xi_2,\tau_1+\tau_2+\Omega^a) \\
&\lesssim \int d\tau_2 \int (d\xi_2)_1 f_{k_2,j_2}^{\#}(\xi_2,\tau_2) \int d\tau_1 (1+2^{j_3-(a-1)k_1-k_2})^{1/2} \\
&\times \left( \int (d\xi_1)_1 |f_{k_1,j_1}^{\#}(\xi_1,\tau_1)|^2 |f_{k_3,j_3}^{\#}(\xi_1+\xi_2,\tau_1+\tau_2+\Omega^a)|^2 \right)^{1/2}
\end{split}
\end{equation*}
Now apply Cauchy-Schwarz like above in $\tau_2$, $\xi_2$ and $\tau_1$ to find
\begin{equation*}
\lesssim 2^{j_1/2} (1+2^{j_3-(a-1)k_1-k_2})^{1/2} \prod_{i=1}^3 \Vert f_{k_i,j_i}\Vert_2
\end{equation*}
This proves the third bound.
\end{proof}
\begin{remark}
Unless one introduces modulation weights like e.g. in \cite{GuoPengWangWang2011} the third bound is insufficient to overcome the derivative loss in case of $High \times Low\rightarrow High$-interaction. Moreover, it is this estimate which complicates short time bilinear estimates in negative Sobolev spaces.
\end{remark}
\begin{lemma}
\label{lem:L2BilinearEstimateHighHighHighInteractionBenjaminOno}
Let $1 \leq a \leq 2$. If $|k_i-k_j| \leq 5$, $i=1,2,3$, then we find \eqref{eq:L2BilinearEstimateBenjaminOno} to hold with $\alpha(\underline{k},\underline{j}) = 2^{j_{i_1}/2} (1+2^{j_{i_2}-(a-1)k_1})^{1/4}$ for any $i_1, i_2 \in \{1,2,3\}$ provided that $i_1 \neq i_2$.\\
Suppose in addition that $||\xi_{i_1}|^a - |\xi_{i_2}|^a| \sim 2^{ak_1}$ provided that $\xi_{i_m} \in \text{supp} (f_{k_{i_m},j_{i_m}})$, $i_m \in \{1,2,3 \}$. Then, we find \eqref{eq:L2BilinearEstimateBenjaminOno} to hold with $\alpha = 2^{j_{i_1}/2} (1+2^{j_{i_2}-ak_1})^{1/2}$.
\end{lemma}
\begin{proof}
We assume in the following that $a>1$ because the claim is covered in Lemma \ref{lem:L4StrichartzEstimateFractionalBenjaminOno} for $a=1$. For the first claim we apply Cauchy-Schwarz in $\xi_2$ to find
\begin{equation*}
\begin{split}
&\int d\tau_1 \int (d\xi_1)_1 f_{k_1,j_1}^{\#}(\xi_1,\tau_1) \int d\tau_2 \int (d\xi_2)_1 f_{k_2,j_2}^{\#}(\xi_2,\tau_2) f_{k_3,j_3}^{\#}(\xi_1+\xi_2,\tau_1+\tau_2+\Omega^a) \\
&\lesssim \int d\tau_1 \int (d\xi_1)_1 f_{k_1,j_1}^{\#}(\xi_1,\tau_1) \int d\tau_2 (1+2^{j_3-(a-1)k_1})^{1/4} \\
&\times \left( \int (d\xi_2)_1 |f_{k_2,j_2}^{\#}(\xi_2,\tau_2)|^2 |f_{k_3,j_3}^{\#}(\xi_1+\xi_2,\tau_1+\tau_2+\Omega^a)|^2 \right)^{1/2}
\end{split}
\end{equation*}
This estimate follows due to
\begin{equation*}
\left| \frac{\partial^2 \Omega^a}{\partial \xi_2^2} \right| \sim 2^{(a-1) k_1},
\end{equation*}
which is straight-forward from Case-by-Case analysis according to the signs of the involved frequencies.\\
Applications of Cauchy-Schwarz in $\tau_1$, $\xi_1$ and $\tau_2$ lead to
\begin{equation*}
\lesssim 2^{j_2/2} (1+2^{j_3-(a-1)k_1})^{1/4} \prod_{i=1}^3 \Vert f_{k_i,j_i} \Vert_2
\end{equation*}
which proves the first claim for $m_1=2$, $m_2=3$. There is no loss of generality due to the symmetry among $k_i$, $i=1,2,3$.\\
For the second claim we argue like in Lemma \ref{lem:L2BilinearEstimateHighLowHighInteractionBenjaminOno}: Let $i_1 = 3$, $i_2=2$. From the proof it will be clear that this is no loss of generality.\\
We apply the Cauchy-Schwarz inequality in $\xi_2$ to find
\begin{equation*}
\begin{split}
&\int d\tau_1 \int (d\xi_1)_1 f_{k_1,j_1}^{\#}(\xi_1,\tau_1) \int d\tau_2 \int (d\xi_2)_1 f_{k_2,j_2}^{\#}(\xi_2,\tau_2) f_{k_3,j_3}^{\#}(\xi_1+\xi_2,\tau_1+\tau_2+\Omega^a) \\
&\lesssim \int d\tau_1 \int (d\xi_1)_1 f_{k_1,j_1}^{\#}(\xi_1,\tau_1) \int d\tau_2 (1+2^{j_3-ak_1})^{1/2} \\
&\times \left( \int (d\xi_2)_1 |f_{k_2,j_2}^{\#}(\xi_2,\tau_2)|^2 |f_{k_3,j_3}^{\#}(\xi_1+\xi_2,\tau_1+\tau_2+\Omega^a)|^2 \right)^{1/2}
\end{split}
\end{equation*} 
Now the claim follows from application of Cauchy-Schwarz inequality in $\tau_1$, $\xi_1$ and $\tau_2$.
\end{proof}
To estimate lower order terms, we use the following estimate not exploiting the dispersion relation, but following from Cauchy-Schwarz inequality:
\begin{lemma}
\label{lem:L2BilinearEstimateCauchySchwarzBenjaminOno}
Estimate \eqref{eq:L2BilinearEstimateBenjaminOno} holds with $\alpha = 2^{k_{\min}/2} 2^{j_{\min}/2}$.
\end{lemma}
\section{Short time bilinear estimates}
\label{section:ShorttimeBilinearEstimates}
Purpose of this section is to prove the following proposition:
\begin{proposition}
\label{prop:ShorttimePropagationNonlinearityBenjaminOno}
Let $T \in (0,1]$ and $u,v \in F^{s,\delta}_a(T)$, $i=1,2$.\\
If $1 < a \leq 3/2$, then there are $\delta = \delta(a,s)>0$ and $\theta = \theta(a,s)>0$, so that we find the following estimates to hold:
\begin{align}
\label{eq:ShorttimeL2EstimateBenjaminOnoSmallDispersion}
\Vert \partial_x (u v) \Vert_{N_a^{0,\delta}(T)} &\lesssim T^\theta \Vert u \Vert_{F^{0,\delta}_a(T)} \Vert v \Vert_{F^{0,\delta}_a(T)} \\
\label{eq:ShorttimeNegativeSobolevEstimateBenjaminOnoSmallDispersion}
\Vert \partial_x (u v) \Vert_{N^{-1/2,\delta}_a(T)} &\lesssim T^\theta \Vert u \Vert_{F^{-1/2,\delta}_a(T)} \Vert v \Vert_{F^{s,\delta}_a(T)}
\end{align}
provided that $s>3/2-a$.\\
If $3/2 < a < 2$, then there are $\delta(a)>0$ and $\varepsilon(a) > 0$, so that we find the following estimate to hold:
\begin{align}
\label{eq:ShorttimeL2EstimateBenjaminOnoLargeDispersion}
\Vert \partial_x (u v) \Vert_{N_a^{0,\delta}(T)} &\lesssim T^\theta \Vert u \Vert_{F^{0,\delta}_a(T)} \Vert v \Vert_{F^{0,\delta}_a(T)} \\
\label{eq:ShorttimeNegativeSobolevEstimateBenjaminOnoLargeDispersion}
\Vert \partial_x ( uv) \Vert_{N^{-1/2,\delta}_a(T)} &\lesssim T^\theta \Vert u \Vert_{F^{0,\delta}_a(T)} \Vert v \Vert_{F^{-1/2,\delta}_a(T)}
\end{align}
\end{proposition}
We will work with $\delta = 0$ in the following which will be omitted from notation. Later we shall see how the analysis yields the estimates claimed above.\\
The above estimates will be proved after decompositions in the frequency, essentially reducing the estimates to
\begin{equation}
\label{eq:ShorttimeFrequencyLocalizedEstimateBenjaminOno}
\Vert P_{k_3} \partial_x (u_{k_1} u_{k_2}) \Vert_{N_{a,k_3}} \lesssim \alpha(\underline{k}) \Vert u_{k_1} \Vert_{F_{a,k_1}} \Vert u_{k_2} \Vert_{F_{a,k_2}}
\end{equation}
These estimates will be proved via the $L^2$-bilinear estimates from the previous section. We enumerate the possible frequency interactions:
\begin{enumerate}
\item[(i)] $High \times Low \rightarrow High$-interaction: This case will be treated in Lemma \ref{lem:HighLowHighShorttimeNonlinearEstimateBenjaminOno}.
\item[(ii)] $High \times High \rightarrow High$-interaction: This case will be treated in Lemma \ref{lem:HighHighHighShorttimeNonlinearEstimateBenjaminOno}.
\item[(iii)] $High \times High \rightarrow Low$-interaction: This case will be treated in Lemma \ref{lem:HighHighLowShorttimeNonlinearEstimateBenjaminOno}.
\item[(iv)] $Low \times Low \rightarrow Low$-interaction: This will be treated in Lemma \ref{lem:LowLowLowShorttimeNonlinearEstimateBenjaminOno}.
\end{enumerate}
We start with $High \times Low \rightarrow High$-interaction:
\begin{lemma}
\label{lem:HighLowHighShorttimeNonlinearEstimateBenjaminOno}
Let $1 \leq a \leq 2$. Suppose that $k_3 \geq 20$, $|k_1-k_3| \leq 5$, $k_2 \leq k_3-10$. Then, we find \eqref{eq:ShorttimeFrequencyLocalizedEstimateBenjaminOno} to hold with $\alpha = 1$.
\end{lemma}
\begin{proof}
Let $\gamma : \R \rightarrow [0,1]$ be a smooth function with $\text{supp} (\gamma) \subseteq [-1,1]$ and
\begin{equation*}
\sum_{m \in \mathbb{Z}} \gamma^2(x-m) \equiv 1
\end{equation*}
Plugging in the definition of $N_{a,k_3}$ we find the lhs in \eqref{eq:ShorttimeFrequencyLocalizedEstimateBenjaminOno} to be dominated by
\begin{equation}
\label{eq:FunctionSpaceReductionShorttimeEstimate}
\begin{split}
&C 2^{k_3} \sum_{m \in \mathbb{Z}} \sup_{t_{k_3} \in \R} \Vert (\tau - \varphi_a(\xi) + i 2^{(2-a)k_3})^{-1} 1_{I_{k_3}}(\xi) \\
&( \mathcal{F}_{t,x}[\eta_0(2^{(2-a)k_3}(t-t_{k_3})) \gamma(2^{(2-a)k_3+10}(t-t_{k_3})-m) u_1]) \\
&*(\mathcal{F}_{t,x}[\gamma(2^{(2-a)k_3+10}(t-t_{k_3})-m) u_2]) \Vert_{X_{a,k_3}}
\end{split}
\end{equation}
Note that
\begin{equation*}
\# \{ m \in \mathbb{Z} | \eta_0(2^{(2-a)k_3}(\cdot - t_k)) \gamma(2^{(2-a)k_3+10}(\cdot - t_{k}) -m) \neq 0 \} = O(1)
\end{equation*}
Consequently, it is enough to estimate \eqref{eq:FunctionSpaceReductionShorttimeEstimate} for fixed $m$. Write
\begin{equation*}
\begin{split}
f_{k_1} &= \mathcal{F}_{t,x}[\eta_0(2^{(2-a)k_3}(t-t_{k_3})) \gamma(2^{(2-a)k_3+10}(t-t_{k_3})-m) u_{k_1}] \\
f_{k_2} &= \mathcal{F}_{t,x}[\gamma(2^{(2-a)k_3+10}(t-t_{k_3})-m) u_{k_2}]
\end{split}
\end{equation*}
Here, we omit dependence on $t_{k_3}$ and $m$ because the derived bounds are uniform in these parameters.\\
Further, we perform an additional localization in modulation
\begin{equation*}
f_{k_i,j_i} = 
\begin{cases}
&\eta_{\leq j_i} (\tau - \varphi_a(\xi)) f_{k_i}. \quad j_i = (2-a)k_3 + 10 \\
&\eta_{j_i}(\tau - \varphi_a(\xi) f_{k_i}, \quad j_i > (2-a)k_3+10 
\end{cases}
\end{equation*}
and by the definition of the $F_{a,k_i}$-spaces and \eqref{eq:XkEstimateIII} we reduce \eqref{eq:ShorttimeFrequencyLocalizedEstimateBenjaminOno} in the peculiar case of $High \times Low \rightarrow High$-interaction to
\begin{equation}
\label{eq:HighLowHighShorttimeModulationLocalization}
2^{k_3} \sum_{j_3 \geq (2-a)k_3} 2^{-j_3/2} \Vert 1_{D_{k_3,j_3}^a} (f_{k_1,j_1} * f_{k_2,j_2}) \Vert_{L^2} \lesssim \prod_{i=1}^2 2^{j_i/2} \Vert f_{k_i,j_i} \Vert_{L^2},
\end{equation}
where $\text{supp} (f_{k_i,j_i}) \subseteq D^a_{k_i,\leq j_i}$ for $i=1,2$ and we can suppose that $j_i \geq (2-a)k_3$.\\
For the resonance function we have the estimate from below
\begin{equation*}
|\Omega^a| \gtrsim 2^{ak_3+k_2}
\end{equation*}
Consequently, there is $j_i \geq ak_3+k_2-10$.\\
Suppose that $j_3 \geq ak_3+k_2-10$. Then, we apply duality and the first bound from Lemma \ref{lem:L2BilinearEstimateHighLowHighInteractionBenjaminOno} to find
\begin{equation}
\label{eq:HighLowHighShorttimeNonlinearReductionI}
\begin{split}
&\sum_{j_3 \geq ak_3+k_2-10} 2^{-j_3/2} \Vert 1_{D_{k_3,j_3}^a} (f_{k_1,j_1} * f_{k_2,j_2}) \Vert_{L^2} \\
&\lesssim 2^{-(ak_3+k_2)/2} 2^{j_2/2} (1+2^{j_1-ak_3})^{1/2} \prod_{i=1}^2 \Vert f_{k_i,j_i} \Vert_2
\end{split}
\end{equation}
By the lower bound for $j_1$ and $a \geq 1$ it follows
\begin{equation*}
\begin{split}
\eqref{eq:HighLowHighShorttimeNonlinearReductionI} \lesssim 2^{-(ak_3+k_2)/2} 2^{j_2/2} 2^{j_1/2} 2^{-(2-a)k_3/2} \prod_{i=1}^2 \Vert f_{k_i,j_i} \Vert_2 \lesssim 2^{-k_2/2} \prod_{i=1}^2 2^{j_i/2} \Vert f_{k_i,j_i} \Vert_{L^2}
\end{split}
\end{equation*}
which yields \eqref{eq:HighLowHighShorttimeModulationLocalization}.\\
Suppose that $j_1 \geq ak_3+k_2-10$. The argument for $j_2 \geq ak_3+k_2-10$ is the same. An application of the second bound from Lemma \ref{lem:HighLowHighShorttimeNonlinearEstimateBenjaminOno} yields
\begin{equation*}
\begin{split}
&\sum_{j_3 \geq (2-a)k_3} 2^{-j_3/2} \Vert 1_{D_{k_3,j_3}^a} (f_{k_1,j_1} * f_{k_2,j_2}) \Vert_{L^2} \\
&\lesssim \sum_{j_3 \geq (2-a)k_3} 2^{-j_3/2} (1+2^{j_1-ak_3})^{1/2} 2^{j_2/2} \prod_{i=1}^2 \Vert f_{k_i,j_i} \Vert_2 \\
&\lesssim 2^{-(2-a)k_3/2} 2^{-ak_3/2} \prod_{i=1}^2 2^{j_i/2} \Vert f_{k_i,j_i} \Vert_2 = 2^{-k_3} \prod_{i=1}^2 2^{j_i/2} \Vert f_{k_i,j_i} \Vert_2
\end{split}
\end{equation*}
This completes the proof.
\end{proof}
We turn to $High \times High \rightarrow High$-interaction:
\begin{lemma}
\label{lem:HighHighHighShorttimeNonlinearEstimateBenjaminOno}
Let $1 \leq a \leq 2$. Suppose that $k_3 \geq 20$, $|k_1-k_2| \leq 5$, $|k_2-k_3| \leq 5$. Then, we find \eqref{eq:ShorttimeFrequencyLocalizedEstimateBenjaminOno} to hold with $\alpha = 1$.
\end{lemma}
Actually, the same argument like in $High \times Low \rightarrow High$-interaction is applicable since there will be two frequencies with group velocity difference of size $2^{ak_1}$. Below we point out how to derive improved estimates using the resonance.
\begin{proof}
Like above it will suffice to prove
\begin{equation}
\label{eq:HighHighHighShorttimeModulationLocalization}
2^{k_3} \sum_{j_3 \geq (2-a)k_3} 2^{-j_3/2} \Vert 1_{D_{k_3,j_3}^a} (f_{k_1,j_1} * f_{k_2,j_2}) \Vert_{L^2} \lesssim \prod_{i=1}^2 2^{j_i/2} \Vert f_{k_i,j_i} \Vert_2
\end{equation}
In this case we have $|\Omega^a| \gtrsim 2^{(a+1)k_3}$. Hence, due to otherwise impossible modulation interaction, there is $j_i \geq (a+1)k_3-10$.\\
If $j_3 \geq (a+1)k_3 - 10$ we use duality and the first estimate from Lemma \ref{lem:L2BilinearEstimateHighHighHighInteractionBenjaminOno} to find
\begin{equation*}
\begin{split}
&\sum_{j_3 \geq (a+1)k_3-10} 2^{-j_3/2} \Vert 1_{D_{k_3,j_3}^a}(f_{k_1,j_1} * f_{k_2,j_2}) \Vert_{L^2} \\
&\lesssim 2^{-(a+1)k_3/2} 2^{j_1/2} (1+2^{j_2-(a-1)k_3})^{1/4} \prod_{i=1}^2 \Vert f_{k_i,j_i} \Vert_2,
\end{split}
\end{equation*}
from which the claim follows even with extra smoothing.\\
If $j_1 \geq (a+1)k_3 - 10$ (or $j_2 \geq (a+1)k_3-10$ where the same estimate can be applied) we use again the first estimate from Lemma \ref{lem:L2BilinearEstimateHighHighHighInteractionBenjaminOno} to derive
\begin{equation*}
\begin{split}
&\sum_{j_3 \geq (2-a)k_3} 2^{-j_3/2} \Vert 1_{D_{k_3,j_3}^a}(f_{k_1,j_1} * f_{k_2,j_2}) \Vert_{L^2} \\
&\lesssim \sum_{j_3 \geq (2-a)k_3} 2^{-j_3/2} (1+2^{j_3-(a-1)k_3})^{1/4} 2^{j_2/2} \prod_{i=1}^2 \Vert f_{k_i,j_i} \Vert_2 \\
&\lesssim 2^{-(1+\varepsilon(a))k_3} \prod_{i=1}^2 2^{j_i/2} \Vert f_{k_i,j_i} \Vert_2
\end{split}
\end{equation*}
even for some $\varepsilon = \varepsilon(a) > 0$.
\end{proof}
We turn to $High \times High \rightarrow Low$-interaction, which is dual to $High \times Low \rightarrow High$-interaction. We have to add localization in time in order to estimate the input functions in short time spaces.
\begin{lemma}
\label{lem:HighHighLowShorttimeNonlinearEstimateBenjaminOno}
Let $k_1 \geq 20$ and $k_3 \leq k_1-10$. Then, we find \eqref{eq:ShorttimeFrequencyLocalizedEstimateBenjaminOno} to hold with $\alpha=(3k_1) 2^{(1-a)k_1} 2^{(a-3/2)k_3}$.
\end{lemma}
\begin{proof}
Following the definition of the $N_{a,k}$-spaces we have to estimate
\begin{equation}
\label{eq:HighHighLowShorttimeInteractionBenjaminOno}
2^{k_3} \sum_{j_3 \geq (2-a)k_3} 2^{-j_3/2} \Vert 1_{D_{k_3,j_3}^a} \mathcal{F}_{t,x} (u_{k_1} v_{k_2} \eta(2^{(2-a)k_3}(t-t_k)) \Vert_{L_{\tau,\xi}^2}, 
\end{equation}
The resonance is given by $|\Omega^a| \gtrsim 2^{ak_1+k_3}$.\\
Suppose that $j_3 \geq ak_1 + k_3 -10$. Then, we find
\begin{equation*}
\eqref{eq:HighHighLowShorttimeInteractionBenjaminOno} \lesssim 2^{k_3} 2^{-\frac{ak_1+k_3}{2}} \Vert u_{k_1} v_{k_2} \eta(2^{(2-a)k_3}(t-t_k)) \Vert_{L^2_{t,x}}
\end{equation*}
After adding localization in time (since we are estimating an $L_t^2$-norm at this point) it is enough to estimate
\begin{equation}
\label{eq:ReductionHighHighLowInteractionFractionalBenjaminOno}
2^{k_3} 2^{\frac{(2-a)(k_1-k_3)}{2}} 2^{-\frac{ak_1+k_3}{2}} \Vert u_{k_1} v_{k_2} \eta(2^{(2-a)k_1}(t-t_\lambda)) \Vert_{L^2_{t,x}}
\end{equation}
Write
\begin{align*}
f_{k_1,j_1} &= 1_{D^a_{k_1,(\leq) j_1}} \mathcal{F}_{t,x}[ \gamma(2^{(2-a)k_1}(t-t_\lambda)) u_{k_1}] \\
f_{k_2,j_2} &= 1_{D^a_{k_2,(\leq) j_2}} \mathcal{F}_{t,x}[ \gamma(2^{(2-a)k_1 + 10}(t-t_\mu))v_{k_2}]
\end{align*}
where the low modulations are annexed matching time localization as usual.\\
Then an application of two $L^4_{t,x}$-Strichartz estimates gives
\begin{equation*}
\begin{split}
\eqref{eq:ReductionHighHighLowInteractionFractionalBenjaminOno} &\lesssim 2^{(1-a)k_1} 2^{\frac{a-1}{2}k_3} 2^{-\frac{k_1}{4}} \prod_{i=1}^2 \sum_{j_i \geq (2-a)k_1} 2^{j_i/2} \Vert f_{k_i,j_i} \Vert_2 \\
&\lesssim 2^{( 3/4 - a ) k_1} 2^{\frac{a-1}{2}k_3} \prod_{i=1}^2 \sum_{j_i \geq (2-a)k_1} 2^{j_i/2} \Vert f_{k_i,j_i} \Vert_2
\end{split}
\end{equation*}
which yields a first bound. Some of the above estimates are crude because the next case gives the worse bound anyway.\\
We turn to the sum over $j_3$ in \eqref{eq:HighHighLowShorttimeInteractionBenjaminOno}, where $j_3 \leq a k_1 + k_3 -10$.\\
By the above reductions and notation we have to estimate
\begin{equation*}
2^{(2-a)(k_1-k_3)} 2^{k_3} \sum_{(2-a)k_3 \leq j_3 \leq ak_1+k_3} 2^{-j_3/2} \Vert 1_{D^a_{k_3,\leq j_3}} (f_{k_1,j_1} * f_{k_2,j_2}) \Vert_{L^2_{\tau,\xi}},
\end{equation*}
where $j_1,j_2 \geq (2-a)k_1$.\\
Suppose that $j_1 \geq ak_1+k_3-10$. An application of Lemma \ref{lem:L2BilinearEstimateHighLowHighInteractionBenjaminOno} in conjunction with duality gives
\begin{equation*}
\begin{split}
&\lesssim 2^{(2-a)(k_1-k_3)} 2^{k_3} \sum_{j_3 \leq ak_1+k_3} 2^{-j_3/2} 2^{j_3/2} \left( 1 + 2^{j_2-ak_1} \right)^{1/2} \Vert f_{k_1,j_1} \Vert_2 \Vert f_{k_2,j_2} \Vert_2 \\
&\lesssim (3k_1) 2^{(1-a)k_1} 2^{(a-3/2)k_3} \prod_{i=1}^2 2^{j_i/2} \Vert f_{k_i,j_i} \Vert_2,
\end{split}
\end{equation*}
which is inferior to the first bound. The proof is complete.
\end{proof}
We record the estimate for $Low \times Low \rightarrow Low$-interaction which is immediate from Lemma \ref{lem:L2BilinearEstimateCauchySchwarzBenjaminOno}:
\begin{lemma}
\label{lem:LowLowLowShorttimeNonlinearEstimateBenjaminOno}
Let $k_i \leq 100$, $i=1,2,3$. Then, we find \eqref{eq:ShorttimeFrequencyLocalizedEstimateBenjaminOno} to hold with $\alpha(\underline{k}) = 1$.
\end{lemma}
\begin{proof}[Proof of Proposition \ref{prop:ShorttimePropagationNonlinearityBenjaminOno}]
With the above estimates for frequency localized interactions at disposal, we can infer the claimed estimates: For $High \times Low \rightarrow High$-interaction Lemma \ref{lem:HighLowHighShorttimeNonlinearEstimateBenjaminOno} gives the estimates after square-summing
\begin{align*}
\Vert \partial_x (uv) \Vert_{N^0_a(T)} &\lesssim \Vert u \Vert_{F^0_a(T)} \Vert v \Vert_{F^{0+}_a(T)} \\
\Vert \partial_x (uv) \Vert_{N_a^{-1/2}(T)} &\lesssim \Vert u \Vert_{F_a^{-1/2}(T)} \Vert v \Vert_{F_a^s(T)}
\end{align*}
where $1 < a \leq 3/2$ and $s>3/2-a$.\\
Increasing time localization leads to extra smoothing (because the minimal size of the modulation regions will become larger) and together with Lemma \ref{lem:tradingModulationRegularity} we deduce from the proof of Lemma \ref{lem:HighLowHighShorttimeNonlinearEstimateBenjaminOno}
\begin{align*}
\Vert \partial_x(uv) \Vert_{N^{0,\delta}_a(T)} &\lesssim T^\varepsilon \Vert u \Vert_{F^{0,\delta}_a(T)} \Vert v \Vert_{F^{0,\delta}_a(T)} \\
\Vert \partial_x(uv) \Vert_{N_a^{-1/2,\delta}(T)} &\lesssim T^\varepsilon \Vert u \Vert_{F_a^{-1/2,\delta}(T)} \Vert v \Vert_{F^{s,\delta}_a(T)}
\end{align*}
for some $\varepsilon > 0$ for any $\delta > 0$ with $a$ and $s$ like in the previous display.\\
For $3/2 < a<2$ the argument is analoguous for $High \times Low \rightarrow High$-interaction.\\
For $High \times High \rightarrow Low$-interaction the short time estimates get worse when increasing time localization. But there is room in the estimate from Lemma \ref{lem:HighHighLowShorttimeNonlinearEstimateBenjaminOno} to prove the estimates for $\delta(a)>0$ chosen sufficiently small.
\end{proof}
\section{Energy estimates}
\label{section:EnergyEstimates}
Purpose of this section is to propagate the energy norm of solutions and differences of solutions: Set
\begin{equation*}
\Vert u \Vert_{H^m}^2 = \sum_{\xi} m(\xi) \hat{u}(\xi) \hat{u}(-\xi)
\end{equation*}
We will consider generalized symbols $m \in S^s_{\varepsilon}$ following \cite{KochTataru2007}: 
\begin{definition}
Let $s \in \R$ and $\varepsilon > 0$. $S^s_{\varepsilon}$ denotes the class of spherically symmetric smooth functions (symbols), where $m \in S^s_\varepsilon$ satisfies
\begin{enumerate}
\item[(i)] symbol regularity,
\begin{equation*}
|\partial^\alpha m| \lesssim_{\alpha} (1+\xi^2)^{-\alpha/2} m(\xi)
\end{equation*}
\item[(ii)] decay at infinity,
\begin{equation*}
s-\varepsilon \leq \frac{\ln m(\xi)}{\ln(1+\xi^2)} \leq s+\varepsilon, \quad \quad s-\varepsilon \leq \frac{d \ln m(\xi)}{d \ln (1+\xi^2)} \leq s+\varepsilon
\end{equation*}
\end{enumerate}
\end{definition}
Also, the proof of Proposition \ref{prop:PropagationEnergyNormsSolutionsFractionalBenjaminOno} based on Lemma \ref{lem:BoundaryTermFractionalBenjaminOno} and Lemma \ref{lem:RemainderEnergyEstimateSolutionsFractionalBenjaminOno} is standard and will be omitted.\\
The following estimates will be shown:
\begin{proposition}
\label{prop:PropagationEnergyNormsSolutionsFractionalBenjaminOno}
Let $1<a<2$, $T \in (0,1]$, $M \in 2^{\mathbb{N}_0}$ and suppose that $u$ is a smooth solution to \eqref{eq:GeneralizedFractionalBenjaminOnoEquation} with vanishing mean. Then, there are positive $\varepsilon(s,a)$, $\theta(a,s)$, $\delta(a,s)$, $c(a,s)$, $d(a,s)$ so that we find the following estimate to hold 
\begin{equation}
\label{eq:EnergyEstimateSolutionFractionalBenjaminOno}
\begin{split}
\Vert u \Vert_{E^s(T)}^2 &\lesssim \Vert u \Vert_{H^s}^2 + T M^{c} \Vert u \Vert^3_{F^{s-\varepsilon,\delta}_{a}(T)} \\
&+  M^{-d} \Vert u \Vert^3_{F^{s-\varepsilon,\delta}_a(T)} + T^{\theta} \Vert u \Vert_{F^{s,\delta}_a(T)}^4
\end{split}
\end{equation}
provided that $s \geq 3/2-a$.
\end{proposition}
The following energy estimates for differences of solutions will be proved.
\begin{proposition}
\label{prop:EnergyEstimateDifferencesSolutionFractionalBenjaminOno}
Let $T \in (0,1]$, $1 < a <2$ and $M \in 2^{\mathbb{N}_0}$. Suppose that $s > 3/2 - a$ and $u_i$, $i=1,2$ are smooth solutions to \eqref{eq:GeneralizedFractionalBenjaminOnoEquation} with vanishing mean. Then, there are positive $c(a,s),d(a,s),\theta(a,s),\delta(a,s)$ so that we find the following estimate to hold:
\begin{equation}
\label{eq:EnergyEstimateIDifferencesFractionalBenjaminOno}
\begin{split}
\Vert v \Vert^2_{E^{-1/2}(T)} &\lesssim \Vert v(0) \Vert_{H^{-1/2}}^2 + T M^{c} \Vert v \Vert^2_{F^{-1/2,\delta}_a(T)} ( \Vert u_1 \Vert_{F_a^{s,\delta}(T)} + \Vert u_2 \Vert_{F^{s,\delta}_a(T)} ) \\
&+ M^{-d} \Vert v \Vert^2_{F^{-1/2,\delta}_a(T)} (\Vert u_1 \Vert_{F_a^{s,\delta}(T)} + \Vert u_2 \Vert_{F^{s,\delta}_a(T)}) \\
&+ T^{\theta} \Vert v \Vert^2_{F_a^{-1/2,\delta}(T)} (\Vert u_1 \Vert_{F_a^{s,\delta}(T)}^2 + \Vert u_2 \Vert_{F_a^{s,\delta}(T)}^2)
\end{split}
\end{equation}
Furthermore, the following estimate holds:
\begin{equation}
\label{eq:EnergyEstimateIIDifferencesFractionalBenjaminOno}
\begin{split}
\Vert v \Vert^2_{E^s(T)} &\lesssim \Vert v(0) \Vert_{H^s}^2 + M^{c} T \Vert v \Vert^2_{F^{s,\delta}_a(T)} ( \Vert v \Vert_{F^{s,\delta}_a(T)} + \Vert u_2 \Vert_{F^{s,\delta}_a(T)} ) \\
&+ M^{-d} \Vert v \Vert^2_{F^{s,\delta}_a(T)} ( \Vert v \Vert_{F^{s,\delta}_a(T)} + \Vert u_2 \Vert_{F^{s,\delta}_a(T)} ) \\
&+ T^\theta ( \Vert v \Vert_{F^{s,\delta}_a(T)} \Vert v \Vert_{F^{-1/2,\delta}_a(T)} \Vert u_2 \Vert_{F^{r,\delta}_a(T)} \Vert u_2 \Vert_{F^{s,\delta}_a(T)} \\
&+ \Vert v \Vert^2_{F^{s,\delta}_a(T)} \Vert u_2 \Vert^2_{F^{s,\delta}_a(T)} + \Vert v \Vert^3_{F^{s,\delta}_a(T)} \Vert u_2 \Vert_{F^{s,\delta}_a(T)} ),
\end{split}
\end{equation}
where $r = s+(2-a)$.
\end{proposition}

For smooth solutions we find by the fundamental theorem of calculus and after symmetrization
\begin{equation*}
\begin{split}
\Vert u(t) \Vert_{H^m}^2 &= \Vert u(0) \Vert_{H^m}^2 \\
&+ C \int_0^t \int_{\Gamma_3} (m(\xi_1) \xi_1 + m(\xi_2) \xi_2 + m(\xi_3) \xi_3) \hat{u}(\xi_1) \hat{u}(\xi_2) \hat{u}(\xi_3) d\Gamma_3 ds
\end{split}
\end{equation*}
To integrate by parts like in \cite{KochTataru2007,GuoOh2018} we consider the first resonance function
\begin{equation}
\label{eq:FirstResonanceFractionalBenjaminOno}
\Omega_1^a(\xi_1,\xi_2,\xi_3) = \xi_1 |\xi_1|^a + \xi_2 |\xi_2|^a + \xi_3 |\xi_3|^a \quad (\xi_1,\xi_2,\xi_3) \in \Gamma_3
\end{equation}
A consequence of the mean value theorem is
\begin{equation*}
| \Omega_1^a(\xi_1,\xi_2,\xi_3)| \gtrsim |\xi_{\max}|^a |\xi_{\min}|
\end{equation*}
and thus, the first resonance does not vanish provided that $\xi_i \neq 0$.
Integration by parts becomes possible and we find
\begin{equation*}
\begin{split}
R_3^{s,m} &= \int_0^T ds \sum_{\substack{\xi_1+\xi_2+\xi_3=0,\\ \xi_i \neq 0}} (m(\xi_1) \xi_1 + m(\xi_2) \xi_2 + m(\xi_3) \xi_3) \hat{u}(s,\xi_1) \hat{u}(s,\xi_2) \hat{u}(s,\xi_3) \\
&= \left[ \sum_{\substack{ \xi_1+\xi_2+\xi_3=0,\\ \xi_i \neq 0}} \frac{(m(\xi_1) \xi_1 + m(\xi_2) \xi_2 + m(\xi_3) \xi_3)}{\Omega_1^a(\xi_1,\xi_2,\xi_3)} \hat{u}(t,\xi_1) \hat{u}(t,\xi_2) \hat{u}(t,\xi_3) \right]_{t=0}^T \\
&+ C \int_0^T \sum_{\substack{ \xi_1+\xi_2+\xi_3 = 0, \\ \xi_i \neq 0}} \frac{m(\xi_1) \xi_1 + m(\xi_2) \xi_2 + m(\xi_3) \xi_3}{\xi_1 |\xi_1|^a + \xi_2 |\xi_2|^a + \xi_3 |\xi_3|^a} \hat{u}(t,\xi_1) \hat{u}(t,\xi_2) \\
&\times \xi_3 \sum_{\substack{\xi_3 = \xi_{31} + \xi_{32},\\ \xi_{3i} \neq 0}} \hat{u}(t,\xi_{31}) \hat{u}(t,\xi_{32}) \\
&= B_3^{s,m}(0;T) + R_4^{s,m}(T)
\end{split}
\end{equation*}
Set
\begin{equation*}
b^{s,m}_3(\xi_1,\xi_2,\xi_3) = \frac{m(\xi_1) \xi_1 + m(\xi_2) \xi_2 + m(\xi_3) \xi_3}{\xi_1 |\xi_1|^a + \xi_2 |\xi_2|^a + \xi_3 |\xi_3|^a}
\end{equation*}
The following estimate of the multiplier is a consequence of the mean value theorem and the lower bound for the resonance function:
\begin{lemma}
\label{lem:MultiplierBoundFractionalBenjaminOno}
Let $|\xi_1| \sim |\xi_2| \gtrsim |\xi_3| > 0$. Then, the following estimate holds:
\begin{equation*}
|b_3^{s,m}(\xi_1,\xi_2,\xi_3)| \lesssim  \frac{\max_{i=1,2,3} |m(\xi_i)|}{|\xi_1|^a}
\end{equation*}
\end{lemma}
We collect the low frequencies as
\begin{equation*}
R_3^{s,m,M} = \int_0^T dt \sum_{\substack{\xi_1+\xi_2+\xi_3 =0, \\ 1 \leq |\xi_i| \leq M }} b_3^{s,m}(\xi_1,\xi_2,\xi_3) \hat{u}(t,\xi_1) \hat{u}(t,\xi_2) \hat{u}(t,\xi_3)
\end{equation*}
Following \cite{GuoOh2018} we differentiate by parts only $R_3^{s,m} - R_3^{s,m,M}$ such that one of the initial frequencies is higher than $M$.\\
This leads us to the boundary term $B_3^{s,m,M}$ with one of the frequencies larger than $M$. We have the following lemma:
\begin{lemma}
\label{lem:BoundaryTermFractionalBenjaminOno}
Suppose that $-1/2 < s <1/2$. Then, we find the following estimate to hold for any $1<a<2$, $\delta \geq 0$:
\begin{equation}
B_3^{s,m,M}(0;T) \lesssim M^{-d(s,a)} \Vert u \Vert^3_{F_a^{s,\delta}(T)}
\end{equation}
\end{lemma}
\begin{proof}
Localize frequencies on a dyadic scale, i.e., $P_{k_i} u_i = u_i$ and suppose $k_1 \geq k_2 \geq k_3$ by symmetry. We use the embedding from Lemma \ref{lem:FsEmbedding} to reduce the bound to a bound of Sobolev norms. By Lemma \ref{lem:MultiplierBoundFractionalBenjaminOno} and H\"older in position space we find the estimate for the evaluation at $t$
\begin{equation*}
\begin{split}
&2^{2\varepsilon k_1} \frac{\max(2^{2sk_1},2^{2sk_3})}{2^{ak_1}} \sum_{\substack{ \xi_1+\xi_2+\xi_3 = 0, \\ \xi_i \neq 0, |\xi_1| \geq M}} |\hat{u}_1(t,\xi_1)| |\hat{u}_2(t,\xi_2)| |\hat{u}_3(t,\xi_3)| \\
&\lesssim 2^{2\varepsilon k_1} \frac{\max(2^{2sk_1},2^{2sk_3})}{2^{ak_1}} \Vert P_{k_1} u(t) \Vert_{L^2} \Vert P_{k_2} u(t) \Vert_{L^2} 2^{k_3/2} \Vert P_{k_3} u(t) \Vert_{L^2} 
\end{split}
\end{equation*}
This expression sums up to the claimed estimate.
\end{proof}
The remainder term is symmetrized once again to find (the constraint for the initial frequencies will be omitted because it is not relevant in the following)
\begin{equation*}
\begin{split}
R_4^{s,m} &= C \int_0^T dt \int_{\Gamma_4} d\Gamma_4 (b_3^{s,m}(\xi_1,\xi_2,\xi_{31}+\xi_{32}) - b_3^{s,m}(-\xi_{31},-\xi_{32},\xi_{31}+\xi_{32})) \xi_3 \\
&\times \hat{u}(t,\xi_1) \hat{u}(t,\xi_2) \hat{u}(t,\xi_{31}) \hat{u}(t,\xi_{32})
\end{split}
\end{equation*}
Set 
\begin{equation*}
b_4^{s,m}(\xi_1,\xi_2,\xi_{31},\xi_{32}) = [b_3^{s,m}(\xi_1,\xi_2,\xi_{31}+\xi_{32}) - b_3^{s,m}(-\xi_{31},-\xi_{32},\xi_{31}+\xi_{32})] \xi_3
\end{equation*}
For the second symmetrization we record again by the mean value theorem
\begin{lemma}
\label{lem:SecondMultiplierBoundFractionalBenjaminOno}
With the above notation we find the following estimate to hold:
\begin{equation*}
|b_4^{s,m}(\xi_1,\xi_2,\xi_{31},\xi_{32})| \lesssim \frac{ \max_{i=1,2,3} |m(\xi_i)|}{\max_{i=1,2,3} |\xi_i|^a} |\xi_3^*|
\end{equation*}
where $|\xi_1^*| \geq |\xi_2^*| \geq \ldots$ denotes a decreasing rearrangement of the $\xi_i$, $i=1,2,31,32$.
\end{lemma}
For the more difficult remainder estimate it is important to note that the second symmetrization cancels the second resonance
\begin{equation}
\label{eq:SecondResonanceBenjaminOno}
\Omega_2^a(\xi_1,\xi_2,\xi_3,\xi_4) = \xi_1 |\xi_1|^a + \xi_2 |\xi_2|^a + \xi_3 |\xi_3|^a + \xi_4 |\xi_4|^a \quad (\xi_1,\xi_2,\xi_3,\xi_4) \in \Gamma_4
\end{equation}
Next, an estimate is derived which is effective when estimating expressions involving two high frequencies and two low frequencies provided that the second resonance is non-vanishing. 
\begin{lemma}
\label{lem:HighLowLowHighMultilinearEstimateFractionalBenjaminOno}
Let $k_i,j_i \in \mathbb{N}$ and $f_{k_i,j_i} \in L^2_{\geq 0}(\mathbb{Z} \times \mathbb{R})$ with $\text{supp} (f_{k_i,j_i}) \subseteq D^a_{k_i,\leq j_i}$. Suppose that $|k_1 - k_4| \leq 5$, $k_2 \leq k_3 \leq k_4-10$ and $\text{supp}_{\xi} (f_{k_m,j_m}) \subseteq I_m$, $m=2,3$, $|I_m| \lesssim 2^l$.\\
Then, we find the following estimate to hold:
\begin{equation}
\label{eq:HighLowLowHighMultilinearEstimate}
\begin{split}
&\int d\Gamma_4(\tau) \int d\Gamma_4(\xi) f_{k_1,j_1}(\xi_1,\tau_1) f_{k_2,j_2}(\xi_2,\tau_2) f_{k_3,j_3}(\xi_3,\tau_3) f_{k_4,j_4}(\xi_4,\tau_4) \\
&\lesssim \min(2^{j_1/2},2^{j_2/2}) (1+2^{j_4-ak_4})^{1/2} 2^{j_3/2} 2^{l/2} \prod_{i=1}^4 \Vert f_{k_i,j_i} \Vert_2
\end{split}
\end{equation}
\end{lemma}
\begin{proof}
Like in Section \ref{section:LinearBilinearEstimates} we rewrite and use consecutive applications of Cauchy-Schwarz inequality
\begin{equation*}
\begin{split}
&\int d\Gamma_4(\tau) \int d\Gamma_4(\xi) f_{k_1,j_1}(\xi_1,\tau_1) f_{k_2,j_2}(\xi_2,\tau_2) f_{k_3,j_3}(\xi_3,\tau_3) f_{k_4,j_4}(\xi_4,\tau_4) \\
&= \int d\tau_1 \int (d\xi_1)_1 f_{k_1,j_1}^{\#}(\xi_1,\tau_1) \int d\tau_2 \int (d\xi_2)_2 f_{k_2,j_2}^{\#}(\xi_2,\tau_2) \\
&\times \int d\tau_3 \int (d\xi_3)_1 f_{k_3,j_3}^{\#}(\xi_3,\tau_3) f^{\#}_{k_4,j_4}(-\xi_1-\xi_2-\xi_3,-\tau_1-\tau_2-\tau_3+\Omega_2^a) \\
&\lesssim \int d\tau_1 \int (d\xi_1)_1 f^{\#}_{k_1,j_1}(\xi_1,\tau_1) \int d\tau_2 \int (d\xi_2)_1 f_{k_2,j_2}^{\#}(\xi_2,\tau_2) \\
&\times \int d\tau_3 (1+ 2^{j_4-ak_4})^{1/2} \left( \int (d\xi_3)_1 |f_{k_3,j_3}^{\#}|^2 |f_{k_4,j_4}^{\#}|^2 \right)^{1/2} \\
&\lesssim (1+2^{j_4-ak_4})^{1/2} \int (d\xi_1)_1 \int d\tau_2 \int (d\xi_2)_1 f_{k_2,j_2}^{\#}(\xi_2,\tau_2) \\
&\times \int d\tau_3 \left( \int d\tau_1 |f_{k_1,j_1}^{\#}(\xi_1,\tau_1)|^2 \right)^{1/2} \\
&\left( \int (d\xi_3)_1 |f_{k_3,j_3}^{\#}(\xi_3,\tau_3)|^2 \Vert f_{k_4,j_4}(\xi_1+\xi_2+\xi_3,\cdot) \Vert_{L^2_{\tau}}^2 \right)^{1/2} \\
&\lesssim 2^{l/2} 2^{j_2/2} (1+2^{j_4-ak_4})^{1/2} 2^{j_3/2} \prod_{i=1}^4 \Vert f_{k_i,j_i} \Vert_2
\end{split}
\end{equation*}
which yields the second estimate.\\
Similarly, we find the first estimate by
\begin{equation*}
\begin{split}
&\int d\tau_2 \int (d\xi_2)_1 f_{k_2,j_2}^{\#}(\xi_2,\tau_2) \int d\tau_1 \int (d\xi_1)_1 f_{k_1,j_1}^{\#}(\xi_1,\tau_1) \\
&\times\int (d\xi_3)_1 \int d\tau_3 f_{k_3,j_3}^{\#}(\xi_3,\tau_3) f_{k_4,j_4}^{\#}(\xi_1+\xi_2+\xi_3,\tau_1+\tau_2+\tau_3+\Omega) \\
&\lesssim (1+2^{j_4-ak_4})^{1/2} \int d\tau_2 \int (d\xi_2)_1 f_{k_2,j_2}^{\#}(\xi_2,\tau_2) \int d\tau_1 \int (d\xi_1)_1 f_{k_1,j_1}^{\#}(\xi_1,\tau_1) \\
&\times \int d\tau_3 \left( \int (d\xi_3)_1 |f_{k_3,j_3}^{\#}(\xi_3,\tau_3)|^2 |f_{k_4,j_4}^{\#}(\xi_1+\xi_2+\xi_3,\tau_1+\tau_2+\tau_3+\Omega_2^a)|^2 \right)^{1/2} \\
&\lesssim 2^{j_1/2} 2^{l/2} 2^{j_3/2} (1+2^{j_4-ak_4})^{1/2} \prod_{i=1}^4 \Vert f_{k_i,j_i} \Vert_2
\end{split}
\end{equation*}
\end{proof}
\begin{remark}
Note that the argument is symmetric with respect to the low frequencies $k_2$ and $k_3$ above and the high frequencies $k_1$ and $k_4$. Below we will freely use the estimates obtained from such permutations.
\end{remark}
We record the following short time consequences (i.e. modulations large depending on the frequencies):
\begin{lemma}
\label{lem:HighLowLowHighMultilinearShorttimeEstimateFractionalBenjaminOno}
Suppose that $|k_1-k_4| \leq 5$, $k_3 \leq k_4-10$ and $j_i \geq (2-a)k_i$.\\
Then, we find the following estimate to hold:
\begin{equation}
\label{eq:HighLowLowHighMultilinearShorttimeEstimateIFractionalBenjaminOno}
\begin{split}
&\int d\Gamma_4(\tau) \int d\Gamma_4(\xi) f_{k_1,j_1}(\xi_1,\tau_1) f_{k_2,j_2}(\xi_2,\tau_2) f_{k_3,j_3}(\xi_3,\tau_3) f_{k_4,j_4}(\xi_4,\tau_4) \\
&\lesssim 2^{-k_1} \prod_{i=1}^4 2^{j_i/2} \Vert f_{k_i,j_i} \Vert_2
\end{split}
\end{equation}
provided that $k_2 \leq k_3-5$.\\
Suppose the initial hypothesis and $| k_2 - k_3| \leq 3$. Then, we find the following estimate to hold:
\begin{equation}
\label{eq:HighLowLowHighMultilinearShorttimeEstimateIIFractionalBenjaminOno}
\begin{split}
&\int d\Gamma_4(\tau) \int d\Gamma_4(\xi) f_{k_1,j_1}(\xi_1,\tau_1) f_{k_2,j_2}(\xi_2,\tau_2) f_{k_3,j_3}(\xi_3,\tau_3) f_{k_4,j_4}(\xi_4,\tau_4) \\
&\lesssim 2^{-k_1} 2^{(0k_3)+} \prod_{i=1}^4 2^{j_i/2} \Vert f_{k_i,j_i} \Vert_2
\end{split}
\end{equation}
\end{lemma}
\begin{proof}
The first claim follows from applying Lemma \ref{lem:HighLowLowHighMultilinearEstimateFractionalBenjaminOno} with $l=k_3$ and observing that $j_{\max} \geq ak_1 + k_3-10$.\\
For the second claim, we carry out a decomposition of the expression into $|\Omega^2_a| \sim 2^{ak_1+l}$ which is equivalent to assuming that $|\xi_2 \pm \xi_3| \sim 2^l$.\\
At this point, we can assume that $f_{k_3,j_3}(\cdot,\tau)$ and $f_{k_4,j_4}(\cdot,\tau)$ are supported in intervals $I_m$, $m=2,3$ of length $2^l$.\\
The decompositions $f_{k_i,j_i}^{I_i}$ are almost orthogonal, that is
\begin{equation*}
\sum_{I_i} \Vert f_{k_i,j_i}^{I_i} \Vert_2^2 \lesssim \Vert f_{k_i,j_i} \Vert_2^2
\end{equation*}
and further, supposing that $|\Omega^2_a| \sim 2^{ak_1 + l}$ there are only finitely many intervals $I_3$ such that there is a non-trivial contribution
\begin{equation}
\label{eq:ResonanceLocalizedMultilinearEstimateFractionalBenjaminOno}
\int d\Gamma_4(\tau) \int_{|\Omega^2_a| \sim 2^{ak_1+l}} d\Gamma_4(\xi) f_{k_1,j_1}(\xi_1,\tau_1) f_{k_2,j_2}^{I_2}(\xi_2,\tau_2) f_{k_3,j_3}^{I_3}(\xi_3,\tau_3) f_{k_4,j_4}(\xi_4,\tau_4)
\end{equation}
The localized expression is amenable to the argument yielding the first estimate and so,
\begin{equation*}
\eqref{eq:ResonanceLocalizedMultilinearEstimateFractionalBenjaminOno} \lesssim 2^{-k_1} \left( \prod_{i=1}^4 2^{j_i/2} \right) \Vert f_{k_1,j_1} \Vert_2 \Vert f_{k_2,j_2}^{I_2} \Vert_2 \Vert f_{k_3,j_3}^{I_3} \Vert_2 \Vert f_{k_4,j_4} \Vert_2
\end{equation*}
The claim follows from carrying out the sum over $I_2$ and $I_3$ by almost orthogonality and the sum over $l$, which leads to the $2^{(0k_3)+}$ loss.
\end{proof}
We have the following estimate due to Cauchy-Schwarz inequality to handle lower order terms:
\begin{lemma}
\label{lem:CauchySchwarzEstimateFractionalBenjaminOno}
Let $k_i$, $j_i \in \mathbb{N}$ and $f_{k_i,j_i} \in L^2_{\geq 0}(\mathbb{Z} \times \mathbb{R})$ with $\text{supp}(f_{k_i,j_i}) \subseteq D^a_{k_i,\leq j_i}$ and let $k_1^* \geq \ldots \geq k_4^*$ and $j_1^* \geq \ldots \geq j_4^*$ denote decreasing rearrangements of $k_i$, $j_i$.\\
Then, we find the following estimate to hold:
\begin{equation*}
\begin{split}
&\int d\Gamma_4(\tau) \int d\Gamma_4(\xi) f_{k_1,j_1}(\xi_1,\tau_1) f_{k_2,j_2}(\xi_2,\tau_2) f_{k_3,j_3}(\xi_3,\tau_3) f_{k_4,j_4}(\xi_4,\tau_4) \\
&\lesssim 2^{k_4^*/2} 2^{k_3^*/2} 2^{j_4^*/2} 2^{j_3^*/2} \prod_{i=1}^4 \Vert f_{k_i,j_i} \Vert_2
\end{split}
\end{equation*}
\end{lemma}
 However, if $\Omega_2^a=0$ we find $|\xi_1^*| = |\xi_2^*|$, $|\xi_3^*| = |\xi_4^*|$, where the actual frequencies have opposite signs. Thus, the sum over the frequencies collapses and two applications of Cauchy-Schwarz in the modulation variables give the following:
\begin{lemma}
\label{lem:VanishingResonanceEstimate}
Let $k_i,j_i \in \mathbb{N}$ and $f_{k_i,j_i} \in L^2_{\geq 0}(\mathbb{Z} \times \mathbb{R})$ with $\text{supp} (f_{k_i,j_i}) \subseteq D^a_{k_i,\leq j_i}$. Let $|k_1-k_4| \leq 2$, $|k_3-k_4| \leq 2$ and $k_1 \geq k_3$ and let $j_1^* \geq \ldots \geq j_4^*$ denote a decreasing rearrangement of the $j_i$.\\
Then, we find the following estimate to hold:
\begin{equation*}
\begin{split}
&\int d\Gamma_4(\tau) \int_{\substack{\xi_1+\xi_4=0,\\ \xi_2+\xi_3=0}} d\Gamma_4(\xi) f_{k_1,j_1}(\xi_1,\tau_1) f_{k_2,j_2}(\xi_2,\tau_2) f_{k_3,j_3}(\xi_3,\tau_3) f_{k_4,j_4}(\xi_4,\tau_4) \\ &\lesssim 2^{j_4^*/2} 2^{j_3^*/2} \prod_{i=1}^4 \Vert f_{k_i,j_i} \Vert_2
\end{split}
\end{equation*}
\end{lemma}
In case there is one frequency clearly lower than the remaining three frequencies, the resonance is very favourable and we will make use of the following bound which is a consequence of three $L^6_{t,x}$-Strichartz estimates from Lemma \ref{lem:L6StrichartzEstimateFractionalDispersion}:
\begin{lemma}
\label{lem:MultilinearL6StrichartzEstimate}
Let $k_i,j_i \in \mathbb{N}$ and $f_{k_i,j_i} \in L^2_{\geq 0}(\mathbb{Z} \times \mathbb{R})$ with $\text{supp} (f_{k_i,j_i}) \subseteq D^a_{k_i,\leq j_i}$ and let $j_1^* \geq \ldots \geq j_4^*$ denote a decreasing rearrangement of the $j_i$.\\
Then, we find the following estimate to hold:
\begin{equation*}
\begin{split}
&\int d\Gamma_4(\tau) \int d\Gamma_4(\xi) f_{k_1,j_1}(\xi_1,\tau_1) f_{k_2,j_2}(\xi_2,\tau_2) f_{k_3,j_3}(\xi_3,\tau_3) f_{k_4,j_4}(\xi_4,\tau_4) \\
&\lesssim 2^{-j_1^*/2} 2^{(0 k_{\max})+} \prod_{i=1}^4 2^{j_i/2} \Vert f_{k_i,j_i} \Vert_2
\end{split}
\end{equation*}
\end{lemma}
\begin{proof}
Let $u_i = \mathcal{F}^{-1}_{t,x}[f_{k_i,j_i}]$ denote the inverse Fourier transform and to simplify the notation let $j_1=j_1^*$.\\
Then, changing back to position space and applying H\"older's inequality gives
\begin{equation*}
\begin{split}
&\int d\Gamma_4(\tau) \int d\Gamma_4(\xi) f_{k_1,j_1}(\xi_1,\tau_1) \ldots f_{k_2,j_2}(\xi_2,\tau_2) \\
&= \int dt \int dx u_1(t,x) \ldots u_4(t,x) \\
&\lesssim \Vert u_1 \Vert_{L^2_{t,x}} \prod_{i=2}^4 \Vert u_i \Vert_{L^6_{t,x}} \lesssim \Vert f_{k_1,j_1} \Vert_{L^2_{t,x}} \prod_{i=2}^4 2^{(0k_i)+} 2^{j_i/2} \Vert f_{k_i,j_i} \Vert_2 \\
&\lesssim 2^{-j_1^*/2} 2^{(0k_{\max})+} \prod_{i=1}^4 2^{j_i/2} \Vert f_{k_i,j_i} \Vert_2
\end{split}
\end{equation*}
\end{proof}
Further, we have the following consequence of four $L^4_{t,x}$-Strichartz estimates:
\begin{lemma}
\label{lem:MultilinearL4StrichartzEstimate}
Let $1 \leq a \leq 2$, $k_i,j_i \in \mathbb{N}$ and $f_{k_i,j_i} \in L^2_{\geq 0}(\mathbb{Z} \times \mathbb{R})$ with $\text{supp}(f_{k_i,j_i}) \subseteq D^a_{k_i,\leq j_i}$. Then, we find the following estimate to hold:
\begin{equation*}
\begin{split}
&\int_{\Gamma_4(\tau)} d\Gamma_4(\tau) \int_{\Gamma_4(\xi)} d\Gamma_4(\xi) f_{k_1,j_1}(\xi_1,\tau_1) f_{k_2,j_2}(\xi_2,\tau_2) f_{k_3,j_3}(\xi_3,\tau_3) f_{k_4,j_4}(\xi_4,\tau_4) \\
&\lesssim \prod_{i=1}^4 2^{\frac{(a+2)j_i}{4(a+1)}} \Vert f_{k_i,j_i} \Vert_2
\end{split}
\end{equation*}
\end{lemma}
\begin{proof}
Like in Lemma \ref{lem:MultilinearL6StrichartzEstimate} change to position space and apply H\"older to find
\begin{equation*}
\begin{split}
&\int d\Gamma_4(\tau) \int d\Gamma_4(\xi) f_{k_1,j_1}(\xi_1,\tau_1) \ldots f_{k_4,j_4}(\xi_4,\tau_4) \\
&= \int dt \int dx u_1(t,x) \ldots u_4(t,x) \\
&\lesssim \prod_{i=1}^4 \Vert u_i \Vert_{L^4_{t,x}} \lesssim \prod_{i=1}^4 2^{\frac{(a+2)j_i}{4(a+1)}} \Vert f_{k_i,j_i} \Vert_2
\end{split}
\end{equation*}
The $L^4_{t,x}$-Strichartz estimate is a consequence of Lemma \ref{lem:L4StrichartzEstimateFractionalBenjaminOno}.
\end{proof}
The more involved remainder estimate, where the above multilinear estimates are deployed, is carried out in the following lemma:
\begin{lemma}
\label{lem:RemainderEnergyEstimateSolutionsFractionalBenjaminOno}
Let $1<a<2$ and $T \in (0,1]$. Suppose that $s \geq 3/2-a$. Then, we find the following estimate to hold:
\begin{equation*}
\left| \int_0^T R_4^{m}[u] ds \right| \lesssim T^\theta \Vert u \Vert_{F^{s-\varepsilon,\delta}_a(T)}^4
\end{equation*}
provided that $m \in S^s_\varepsilon$, where $\varepsilon(s,a) > 0$, $\theta(s,a) >0$, $\delta = \delta(s,a)>0$ are chosen sufficiently small.
\end{lemma}
\begin{proof}
In the expression
\begin{equation}
\label{eq:RemainderFractionalBenjaminOno}
\int_0^T dt \int_{\Gamma_4} d\Gamma_4 b^{m_\varepsilon}_4(\xi_1,\xi_2,\xi_{31},\xi_{32}) \hat{u}(\xi_1) \hat{u}(\xi_2) \hat{u}(\xi_{31}) \hat{u}(\xi_{32})
\end{equation}
we can suppose $|\xi_1| \gtrsim |\xi_2|$, $|\xi_{31}| \gtrsim |\xi_{32}|$ by symmetry.\\
Further, we break the frequencies into dyadic blocks $|\xi_1| \sim 2^{k_1}$, $|\xi_2| \sim 2^{k_2}$, $|\xi_{31}| \sim 2^{k_{31}}$, $|\xi_{32}| \sim 2^{k_{32}}$.\\
After dyadic frequency localization for an estimate of \eqref{eq:RemainderFractionalBenjaminOno} one has additionally to take into account the time localization and the multiplier bound. For this purpose, we perform a Case-by-Case analysis:\\
\textbf{Case A.} $|\xi_1| \sim |\xi_2|$\\
\underline{Subcase AI.} $|\xi_1| \gg |\xi_3| \gtrsim |\xi_{31}| \gtrsim |\xi_{32}|$\\
\underline{Subcase AII.} $|\xi_1| \gg |\xi_3| \ll |\xi_{31}| \sim |\xi_{32}|$ \\
\underline{Subcase AIII.} $|\xi_1| \sim |\xi_3| \gtrsim |\xi_{31}| \gtrsim |\xi_{32}|$ \\
\underline{Subcase AIV.} $|\xi_1| \sim |\xi_3| \ll |\xi_{31}| \sim |\xi_{32}|$ \\
\textbf{Case B.} $|\xi_1| \gg |\xi_2|$\\
\underline{Subcase BI.} $|\xi_1| \sim |\xi_3| \sim |\xi_{31}| \sim |\xi_{32}|$\\
\underline{Subcase BII.} $|\xi_1| \sim |\xi_3| \ll |\xi_{31}| \sim |\xi_{32}|$\\
\underline{Subcase BIII.} $|\xi_1| \sim |\xi_3| \sim |\xi_{31}| \gg |\xi_{32}|$\\
Let $\gamma: \R \rightarrow [0,1]$ denote a smooth function with support in $[-1,1]$ satisfying
\begin{equation*}
\sum_{n \in \mathbb{Z}} \gamma^4(x-n) \equiv 1
\end{equation*}
We have
\begin{equation*}
\begin{split}
\eqref{eq:RemainderFractionalBenjaminOno}_{|\xi_1| \sim 2^{k_1},\ldots} &= \sum_{|m| \lesssim T 2^{\alpha k_{\max}}} \int_{\R} dt \int_{\Gamma_4, |\xi_1| \sim 2^{k_1},\ldots} b_4^{m_\varepsilon}(\xi_1,\xi_2,\xi_{31},\xi_{32}) \\
&1_{[0,T]}(t) \gamma(2^{-\alpha k_{\max}} t -m) \hat{u}(\xi_1) \ldots \gamma(2^{-\alpha k_{\max}} t -m) \hat{u}(\xi_{32})
\end{split}
\end{equation*}
where $\alpha = (2-a-\delta)$, so that the products $\gamma(2^{-\alpha k_{\max}} t -m) \hat{u}(\xi_i)$ are estimated in the $F^\delta_{a,k_i}$.\\
Here and below we confine ourselves to the majority of the cases, where the smooth cutoff does not interact with the sharp cutoff, i.e., only the $m \in \mathbb{Z}$ are considered, where
\begin{equation}
\label{eq:NoSharpCutoffFractionalBenjaminOno}
1_{[0,T]}(t) \gamma(2^{-\alpha k_{\max}} t - m) = \gamma(2^{-\alpha k_{\max}} t -m)
\end{equation}
Observe that there are at most four exceptional cases, where the above display fails. These can be treated by interpolation with the estimate from Lemma \ref{lem:CauchySchwarzEstimateFractionalBenjaminOno} and Lemma \ref{lem:sharpTimeCutoffAlmostBounded}.\\
Thus, adapting the reductions and notations from Section \ref{section:ShorttimeBilinearEstimates} one has to estimate
\begin{equation}
\label{eq:EnergyEstimateReductionBenjaminOno}
\begin{split}
&T 2^{(2-a+\delta)k_1^*} |b_4(2^{k_1},2^{k_2},2^{k_{31}},2^{k_{32}})| \int d\Gamma_4(\tau) \int_{\Omega_2^a \neq 0} d\Gamma_4(\xi) f_{k_1,j_1}(\xi_1,\tau_1) \\
&\times f_{k_2,j_2}(\xi_2,\tau_2) f_{k_{31},j_{31}}(\xi_{31},\tau_{31}) f_{k_{32},j_{32}}(\xi_{32},\tau_{32})
\end{split}
\end{equation}
where $j_i \geq (2-a+\delta)k^*_1$, $i=1,2,31,32$ taking into account the time localization. For the sake of brevity write in the following $f_{k_3,j_3} = f_{k_{31},j_{31}}$ and $f_{k_4,j_4} = f_{k_{32},j_{32}}$.\\
For the estimate we will use Lemma \ref{lem:HighLowLowHighMultilinearShorttimeEstimateFractionalBenjaminOno} and \ref{lem:MultilinearL6StrichartzEstimate} in case of separated frequencies and Lemma \ref{lem:L4StrichartzEstimateFractionalBenjaminOno}, whenever the frequencies are not separated. We turn to the single cases.\\
\underline{Subcase AI.} For $b_4^m$ we have the size estimate $|b_4^m| \lesssim \frac{\max( 2^{2sk_1}, 2^{2sk_3}) 2^{2 \varepsilon k_1}}{2^{ak_1}} 2^{k_3}$. The time localization yields a factor of $T 2^{(2-a+\delta)k_1}$ and an application of Lemma \ref{lem:HighLowLowHighMultilinearShorttimeEstimateFractionalBenjaminOno} gives
\begin{equation*}
\eqref{eq:EnergyEstimateReductionBenjaminOno} \lesssim \max( 2^{2s k_1}, 2^{2sk_3}) 2^{k_3-k_1} 2^{2(1-a)k_1} 2^{\delta k_1} 2^{2 \varepsilon k_1} \prod_{i=1}^4 2^{j_i/2} \Vert f_{k_i,j_i} \Vert_2
\end{equation*}
\underline{Subcase AII.} In case the frequencies are not of comparable size one can argue like in Case AI.\\
Otherwise, we apply Lemma \ref{lem:L4StrichartzEstimateFractionalBenjaminOno} to find together with the size estimate of $b^m_4$ and the time localization
\begin{equation*}
\eqref{eq:EnergyEstimateReductionBenjaminOno} \lesssim T \frac{\max( 2^{2s k_1}, 2^{2sk_3})}{2^{ak_1}} 2^{k_3} 2^{(2-a+\delta)k_1} 2^{2 \varepsilon k_1} 2^{-(\delta + 3 \varepsilon) k_1} \prod_{i=1}^4 2^{j_i/2} \Vert f_{k_i,j_i} \Vert_2
\end{equation*}
\underline{Subcase AIII.} This case can be covered following along the above lines.\\
\underline{Subcase AIV.} The size estimate for $b^m_4$ is $|b_4^m| \lesssim \frac{\max( 2^{2s k_1}, 2^{2sk_3})}{2^{ak_1}} 2^{2 \varepsilon k_1} 2^{k_1}$. The time localization yields a factor of $T 2^{(2-a+\delta)k_{31}}$ and an application of Lemma \ref{lem:HighLowLowHighMultilinearShorttimeEstimateFractionalBenjaminOno} gives a smoothing factor of $2^{-k_{31}} 2^{\varepsilon k_{1}}$, which yields
\begin{equation*}
\eqref{eq:EnergyEstimateReductionBenjaminOno} \lesssim T \max( 2^{2s k_1}, 2^{2sk_3}) 2^{(1-a)k_1} 2^{(1-a)k_{31}} 2^{\delta k_{31}} 2^{2\varepsilon k_1} \prod_{i=1}^4 2^{j_i/2} \Vert f_{k_i,j_i} \Vert_2
\end{equation*}
\underline{Subcase BI.} The size estimate of $b_4^m$ is $|b_4^m| \lesssim \frac{\max( 2^{2s k_1}, 2^{2sk_2}) 2^{k_1}}{2^{ak_1}} 2^{2 \varepsilon k_1}$, time localization amounts to a factor of $T2^{(2-a+\delta)k_1}$ and using the resonance $|\Omega_2^a| \gtrsim 2^{(a+1)k_1}$, hence, $j_1^* \geq (a+1)k_1/2 - 10$ in conjunction with Lemma \ref{lem:MultilinearL6StrichartzEstimate} we find
\begin{equation*}
\begin{split}
\eqref{eq:EnergyEstimateReductionBenjaminOno} &\lesssim T \frac{2^{2(s+\varepsilon) k_1}}{2^{ak_1}} 2^{k_1} 2^{-(a+1)k_1/2} 2^{(2-a+\delta)k_1} 2^{3 \varepsilon k_1} \prod_{i=1}^4 2^{j_i/2} \Vert f_{k_i,j_i} \Vert_2 \\
&\lesssim T 2^{2sk_1} 2^{5/2(1-a)} 2^{3 \varepsilon k_1} \prod_{i=1}^4 2^{j_i/2} \Vert f_{k_i,j_i} \Vert_2
\end{split}
\end{equation*}
\underline{Subcase BII.} The size estimate is $|b_4^m| \lesssim \frac{\max( 2^{2s k_1}, 2^{2sk_2}) 2^{2 \varepsilon k_1}}{2^{ak_1}}$, time localization gives a factor of $T 2^{(2-a+\delta)k_{31}}$ and by Lemma \ref{lem:HighLowLowHighMultilinearShorttimeEstimateFractionalBenjaminOno} we find
\begin{equation*}
\eqref{eq:EnergyEstimateReductionBenjaminOno} \lesssim T \frac{\max( 2^{2s k_1}, 2^{2sk_2}) 2^{k_1}}{2^{ak_1}} 2^{(2-a+\delta)k_{31}} 2^{-k_{31}} 2^{3\varepsilon k_{31}} \prod_{i=1}^4 2^{j_i/2} \Vert f_{k_i,j_i} \Vert_{L^2}
\end{equation*}
\underline{Subcase BIII.} 
The size of $b^m_4$ is given by $|b^m_4| \lesssim \frac{\max( 2^{2s k_1}, 2^{2sk_2})}{2^{ak_1}} 2^{(1+2 \varepsilon) k_1}$. Time localization gives a factor of $T2^{(2-a+\delta)k_1}$ and an application of Lemma \ref{lem:HighLowLowHighMultilinearShorttimeEstimateFractionalBenjaminOno} gives
\begin{equation*}
\begin{split}
\eqref{eq:EnergyEstimateReductionBenjaminOno} &\lesssim T \frac{\max( 2^{2s k_1}, 2^{2sk_2}) 2^{k_1}}{2^{ak_1}}  2^{(2-a+\delta)k_1} 2^{3\varepsilon k_1} 2^{-k_1} \prod_{i=1}^4 2^{j_i/2} \Vert f_{k_i,j_i} \Vert_2 \\
&\lesssim T \max( 2^{2s k_1}, 2^{2sk_2}) 2^{2(1-a)k_1} 2^{(3 \varepsilon + \delta) k_1} \prod_{i=1}^4 2^{j_i/2} \Vert f_{k_i,j_i} \Vert_2
\end{split}
\end{equation*}
In all cases we find extra smoothing. It is straight-forward to carry out the summations.
\end{proof}
We turn to the proof of energy estimates for differences of solutions.
\begin{proof}[Proof of Proposition \ref{prop:EnergyEstimateDifferencesSolutionFractionalBenjaminOno}]
We start with the proof of \eqref{eq:EnergyEstimateIDifferencesFractionalBenjaminOno}.\\
An application of the fundamental theorem of calculus gives
\begin{equation*}
\begin{split}
2^{-n} \Vert P_n v(t) \Vert^2_{L^2} &= 2^{-n} \Vert P_n v(0) \Vert^2_{L^2} \\
&+ 2^{-n} 2 \int_0^T dt \sum_{\substack{\xi_1+\xi_2+\xi_3=0,\\ \xi_i \neq 0}} \chi^2_n(\xi_1) \xi_1 \hat{v}(\xi_1) (\hat{u}_1(\xi_2) + \hat{u}_2(\xi_2)) \hat{v}(\xi_3) 
\end{split}
\end{equation*}
In the following we pretend that $v$ is governed by $\partial_t v + \partial_x D_x^a v = \partial_x (vu_1)$ to lighten the notation because we can prove the same estimates replacing $u_1$ with $u_2$ due to multilinearity of the argument.\\
The estimate will be carried out by Case-by-Case analysis which is more involved than in the energy estimates for solutions due to reduced symmetry. For the interaction between $v,u_1,v$ in the above display we have to take care of the following cases:
\begin{enumerate}
\item[Case I]: $High \times Low \rightarrow High$-interaction: $(v,u_1,v)$
\item[Case II]: $High \times Low \rightarrow High$-interaction: $(v,v,u_1)$
\item[Case III]: $High \times High \rightarrow High$-interaction
\item[Case IV]: $High \times High \rightarrow Low$-interaction: $(v,u_1,v)$
\end{enumerate}
We start with an analysis of Case I. After integration by parts and switching back to position space we find
\begin{equation}
\label{eq:EnergyEstimateDifferencesCaseIReduction}
2^{-n} 2^k \int_0^T ds \int dx P_n u P_k u_1 P_{n^\prime} v \quad (|n-n^\prime| \leq 5, k \leq n-10)
\end{equation}
Strictly speaking, the estimates are carried out rather for the absolute values of the space-time Fourier transform which becomes only possible after integration by parts in time first. The above notation is used in order to make the argument more readable.\\
Further, we omit to indicate the summation over the frequencies. One checks that the expressions sum up to the desired regularities.\\
Integration by parts in time is only carried out for $n \geq \log_2(M)$: This gives
\begin{equation*}
\begin{split}
\eqref{eq:EnergyEstimateDifferencesCaseIReduction} &= 2^{-n} 2^k 2^{-(an + k)} \left[ P_n v P_{k} u_1 P_{n^\prime} v \right]_{t=0}^T \\
&+ 2^{-n} 2^k 2^{-(an+k)} ( \int_0^T dt \int \partial_x P_n (v u_1) P_k u_1 P_{n^\prime} v + \int _0^T ds \int P_n v \partial_x P_k (u_1^2) P_{n^\prime} v ) \\
&= B_I(0;T) + I_1 + I_2, \quad |n-n^\prime| \leq 5, k \leq n-6
\end{split}
\end{equation*}
Like in the proof of Proposition \ref{prop:PropagationEnergyNormsSolutionsFractionalBenjaminOno} we only integrate by parts the high frequencies. The boundary term can be estimated using H\"older's inequality and Bernstein's inequality like in the estimate of the boundary term for solutions:
\begin{equation*}
\begin{split}
&\sum_{n \geq m} \sum_{k \leq n-6} \sum_{|n-n^\prime| \leq 5} 2^{-(a+1)n} \int dx P_n v(t) P_k u_1(t) P_{n^\prime} v(t) \\
&\lesssim \sum_{n \geq m} \sum_{k \leq n-6} \sum_{|n-n^\prime| \leq 5} 2^{-(a+1)n} \Vert P_n v(t) \Vert_{L^2} \Vert P_k u_1(t) \Vert_{L^\infty} \Vert P_{n^\prime} v(t) \Vert_{L^2} \\
&\lesssim M^{-d} \Vert v \Vert_{F^{-1/2,\delta}_a(T)}^2 \Vert u_1 \Vert_{F^{s,\delta}_a(T)}
\end{split}
\end{equation*}
where the ultimate estimate follows from Lemma \ref{lem:FsEmbedding}.\\
Moreover, for the low frequencies it is straight-forward to infer by the same means that
\begin{equation*}
\sum_{1 \leq n \leq m} \sum_{k \leq n-6} 2^{-n} 2^k \int_0^T dt \int dx P_n v P_k u_1 P_{n^\prime} v \lesssim T M^{c} \Vert v \Vert^2_{F_a^{-1/2,\delta}(T)} \Vert u_1 \Vert_{F_a^{s,\delta}(T)}
\end{equation*}
We turn to the more involved estimate of $I_1$ and $I_2$. The frequency constraint will be omitted in the following. Compared to the remainder estimate for solutions the multiplier is slightly worse because we do not integrate by parts another time. Moreover, the second resonance can vanish.\\
We split $I_1 = I_{11} + I_{12} + I_{13}$ according to Littlewood-Paley decomposition, which means that we consider $High \times Low \rightarrow High$-interaction for $I_{11}$, $High \times High \rightarrow High$-interaction for $I_{12}$ and $High \times High \rightarrow Low$-interaction for $I_{13}$.\\
If the second resonance does not vanish, then Lemma \ref{lem:HighLowLowHighMultilinearShorttimeEstimateFractionalBenjaminOno} applies and we find
\begin{equation*}
\begin{split}
I_{11} &= 2^{-an} \left| \int_0^T dt \int \left( P_n v P_{k^\prime} u_1 + P_n u_1 P_{k^\prime} v \right) P_{k} u_1 P_{n^\prime} v \right| \\
 &\lesssim T 2^{(2-a+\delta)n} 2^{-an} 2^{-n} 2^{\varepsilon n} \left( \Vert P_n v \Vert_{F^{\delta}_{a,n}} \Vert P_{k^\prime} u_1 \Vert_{F^{\delta}_{a,k^\prime}} + \Vert P_n u_1 \Vert_{F^{\delta}_{a,n}} \Vert P_{k^\prime} v \Vert_{F^{\delta}_{a,k^\prime}} \right) \\
 &\times \Vert P_k u_1 \Vert_{F^{\delta}_{a,k}} \Vert P_{n^\prime} v \Vert_{F^{\delta}_{a,n^\prime}}
\end{split}
\end{equation*}
If the second resonance vanishes, then we use Lemma \ref{lem:VanishingResonanceEstimate} which ameliorates the factor $2^{(2-a+\delta)n}$ from the time localization and gives
\begin{equation*}
I_{11} \lesssim T2^{-an} \left( \Vert P_n v \Vert_{F^{\delta}_{a,n}} \Vert P_{k} u_1 \Vert_{F^{\delta}_{a,k}} + \Vert P_{n} u_1 \Vert_{F^{\delta}_{a,n}} \Vert P_{k} v \Vert_{F^{\delta}_{a,k}} \right) \Vert P_k u_1 \Vert_{F_k} \Vert P_{n} v \Vert_{F^{\delta}_{a,n}}
\end{equation*}
For $I_{12}$ we have to estimate
\begin{equation*}
2^{-an} \int_0^T dt \int P_n v P_{n^\prime} u_1 P_k u_1 P_{n^{\prime \prime}} v dx, \quad k \leq n-10, \; |n-n^\prime| \leq 5, \; |n^{\prime \prime} - n | \leq 5
\end{equation*}
The second resonance satisfies $|\Omega_2^a| \gtrsim 2^{(a+1)n}$. By Lemma \ref{lem:MultilinearL6StrichartzEstimate} we find
\begin{equation*}
\begin{split}
I_{12} &\lesssim T 2^{(2-a+\delta)n} 2^{-an} 2^{-(a+1)n/2} 2^{\varepsilon n} \Vert P_n v \Vert_{F^{\delta}_{a,n}} \Vert P_{n^\prime} u_1 \Vert_{F^{\delta}_{a,n^\prime}} \Vert P_k u_1 \Vert_{F^{\delta}_{a,k}} \Vert P_{n^{\prime \prime}} v \Vert_{F^{\delta}_{a,n^{\prime \prime}}} \\
&\lesssim T 2^{(3/2-5a/2)n} 2^{(\varepsilon +\delta)n} \Vert P_n v \Vert_{F^{\delta}_{a,n}} \Vert P_{n^\prime} u_1 \Vert_{F^{\delta}_{a,n^\prime}} \Vert P_k u_1 \Vert_{F^{\delta}_{a,k}} \Vert P_{n^{\prime \prime}} v \Vert_{F^{\delta}_{a,n^{\prime \prime}}}
\end{split}
\end{equation*}
We turn to $High \times High \rightarrow Low$-interaction: This amounts to estimate
\begin{equation*}
I_{13} = 2^{-an} \int_0^T dt \int P_{m_1} v P_{m_2} u_1 P_k u_1 P_{n^\prime} v \quad (|m_1-m_2| \leq 5, \; n \leq m_1-6)
\end{equation*}
$I_{13}$ is amenable to Lemma \ref{lem:HighLowLowHighMultilinearShorttimeEstimateFractionalBenjaminOno} after adding time localization $T2^{(2-a+\delta)m_1}$ and taking all factors together we find
\begin{equation*}
I_{13} \lesssim T 2^{(1-a)m_1} 2^{(\varepsilon+\delta)m_1} 2^{-an} \Vert P_{m_1} v \Vert_{F^{\delta}_{a,m_1}} \Vert P_{m_2} u_1 \Vert_{F^{\delta}_{a,m_2}} \Vert P_k u_1 \Vert_{F^{\delta}_{a,k}} \Vert P_{n^\prime} v \Vert_{F^{\delta}_{a,n^\prime}}
\end{equation*}
For $I_2$ we use again Littlewood-Paley decomposition to write $I_2 = I_{21} + I_{22} + I_{23}$ like above.\\
Since the deployed arguments are multilinear, the estimates for $I_{21}$ and $I_{22}$ are carried out like above. However, in case of $I_{23}$ we encounter the additional case of comparable frequencies
\begin{equation*}
2^{-n} 2^{-an} 2^{k} \int_0^T dt \int P_n v P_{m_1} u_1 P_{m_2} u_1 P_{n^\prime} v \quad |m_1-m_2| \leq 10, \; |m_1-n| \leq 10
\end{equation*}
which is not necessarily amenable to Lemma \ref{lem:HighLowLowHighMultilinearShorttimeEstimateFractionalBenjaminOno}.\\
But, after adding localization in time $T2^{(2-a+\delta)n}$ and using Lemma \ref{lem:MultilinearL4StrichartzEstimate} in the non-resonant case and Lemma \ref{lem:VanishingResonanceEstimate} in the resonant case we find the estimate
\begin{equation*}
I_{23} \lesssim T 2^{2(1-a)n} 2^{k-n} \Vert P_n v \Vert_{F^{\delta}_{a,n}} \Vert P_{m_1} u_1 \Vert_{F^{\delta}_{a,m_1}} \Vert P_{m_2} u_1 \Vert_{F^{\delta}_{a,m_2}} \Vert P_{n^\prime} v \Vert_{F^{\delta}_{a,n^\prime}},
\end{equation*}
which is again more than enough.\\
In Case II we can not integrate by parts in space to put the derivative on a more favourable factor, thus we have to estimate the expression
\begin{equation}
\label{eq:EnergyEstimateDifferenceCaseIIFractionalBenjaminOno}
\int_0^T dt \int P_n v P_{n^\prime} u_1 P_k v
\end{equation}
Integration by parts in time yields
\begin{equation*}
\begin{split}
II &= 2^{-(an + k)} \left[ P_n v P_{n^\prime} u_1 P_k v \right]_{t=0}^T + 2^{-(an+k)} \left( \int_0^T dt \int \partial_x P_n (v u_1) P_{n^\prime} u_1 P_k v \right. \\
&\left. + \int_0^T dt \int P_n v \partial_x P_{n^\prime} (u_1^2) P_k v + \int_0^T dt \int P_n v P_{n^\prime} u_1 \partial_x P_k(v u_1) \right) \\
&= B_{II}(0;T) + II_1 + II_2 + II_3
\end{split}
\end{equation*}
To derive suitable estimates however, we do not integrate by parts all of \eqref{eq:EnergyEstimateDifferenceCaseIIFractionalBenjaminOno}, but only the part with high frequencies like above. We find for the boundary term with initial frequencies $n \geq \log_2(M)$ following along the above lines of the estimate for $B_I(0;T)$:
\begin{equation*}
B_{II,M}(0;T) \lesssim M^{-c} \Vert v \Vert^2_{F_a^{-1/2,\delta}(T)} \Vert u_1 \Vert_{F^{s,\delta}_a(T)}
\end{equation*}
and for the low frequencies like above
\begin{equation*}
\sum_{1 \leq n \leq m} \sum_{|n-n^\prime| \leq 5} \sum_{k \leq n-6} \int_0^T ds \int dx P_n v P_{n^\prime} u_1 P_k v \lesssim T M^{d} \Vert v \Vert^2_{F^{-1/2}_{a,\delta}(T)} \Vert u_1 \Vert_{F^s_{a,\delta}(T)} 
\end{equation*}
We turn to the estimate of $II_1$. For the evaluation we plug in Littlewood-Paley decomposition of $P_n ( u_1 v)$, and split like above $II_1 = II_{11} + II_{12} + II_{13}$.\\
We have
\begin{equation*}
\begin{split}
II_{11} &= 2^{-(an + k)} 2^n \left( \int_0^T dt \int P_n v P_{k^\prime} u_1 P_{n^\prime} u_1 P_k v + \int_0^T dt \int P_{k^\prime} v P_n u_1 P_{n^\prime} u_1 P_k v \right)\\
&(|n-n^\prime| \leq 5, \quad k,k^\prime \leq n-6)
\end{split}
\end{equation*}
Time localization amounts to a factor of $T 2^{(2-a+\delta)n}$. In the non-resonant case we use Lemma \ref{lem:HighLowLowHighMultilinearShorttimeEstimateFractionalBenjaminOno} and in the resonant case Lemma \ref{lem:VanishingResonanceEstimate} to find gathering all factors
\begin{equation*}
\begin{split}
II_{11} &\lesssim T 2^{-k} 2^{(1-a)n} \Vert P_{n^\prime} u_1 \Vert_{F^{\delta}_{a,n^\prime}} \Vert P_k v \Vert_{F^{\delta}_{a,k}}  \\
&\left( \Vert P_n v \Vert_{F^{\delta}_{a,n}} \Vert P_{k^\prime} u_1 \Vert_{F^{\delta}_{a,k^\prime}} + \Vert P_{k^\prime} u_1 \Vert_{F^{\delta}_{a,k^\prime}} \Vert P_n v \Vert_{F^{\delta}_{a,n}} \right) \Vert P_{n^\prime} u_1 \Vert_{F^{\delta}_{a,n^\prime}} \Vert P_k v \Vert_{F^{\delta}_{a,k}} 
\end{split}
\end{equation*}
For $II_{12}$ we have to estimate
\begin{equation}
\label{eq:EnergyEstimateDifferenceCaseIIReductionII}
2^{(1-a)n-k} \int_0^T dt \int P_{n_1} v P_{n_2} u_1 P_{n^\prime} u_1 P_k v, \quad |n_1-n^\prime| \leq 5, \; |n_2-n^\prime| \leq 5
\end{equation}
For this we use Lemma \ref{lem:MultilinearL6StrichartzEstimate} because the second resonance $|\Omega^2_a| \gtrsim 2^{(a+1)n}$ is favourable:
\begin{equation*}
\eqref{eq:EnergyEstimateDifferenceCaseIIReductionII} \lesssim T 2^{(2-a+\delta)n} 2^{(1-a)n} 2^{-k} 2^{-(a+1)n/2} \Vert P_{n_1} v \Vert_{F^{\delta}_{a,n_1}} \Vert P_{n_2} u_1 \Vert_{F^{\delta}_{a,n_2}} \Vert P_{n^\prime} u_1 \Vert_{F^{\delta}_{a,n^\prime}} \Vert P_k v \Vert_{F^{\delta}_{a,k}}
\end{equation*}
For $II_{13}$ estimate by Lemma \ref{lem:HighLowLowHighMultilinearShorttimeEstimateFractionalBenjaminOno}
\begin{equation*}
\begin{split}
&2^{(1-a)n-k} \int_0^T dt \int P_{m_1} v P_{m_2} u_2 P_{n^\prime} u_2 P_k v \\
&\lesssim T 2^{(1-a+\delta)m_1} 2^{(1-a)n-k} \Vert P_{m_1} v \Vert_{F^{\delta}_{a,m_1}} \Vert P_{m_2} u_2 \Vert_{F^{\delta}_{a,m_2}} \Vert P_{n^\prime} u_2 \Vert_{F^{\delta}_{a,n^\prime}} \Vert P_k v \Vert_{F^{\delta}_{a,k}},
\end{split}
\end{equation*}
where $|m_1-m_2| \leq 5$, $n^\prime \leq m_1-6$.\\
Like above split $II_2 = II_{21} + II_{22} + II_{23}$ and for $II_{21}$ we have to estimate
\begin{equation*}
2^{(1-a)n-k} \int_0^T dt \int P_n v P_{n^\prime} u_1 P_{k^\prime} u_1 P_k v \quad (|n-n^\prime| \leq 3, k, k^\prime \leq n-6)
\end{equation*}
In the non-resonant case we find by applying Lemma \ref{lem:HighLowLowHighMultilinearEstimateFractionalBenjaminOno}
\begin{equation*}
II_{21} \lesssim T 2^{2(1-a)n} 2^{(\delta + \varepsilon)n} 2^{-k} \Vert P_n v \Vert_{F^{\delta}_{a,n}} \Vert P_{n^\prime} u_1 \Vert_{F^{\delta}_{a,n^\prime}} \Vert P_{k^\prime} u_1 \Vert_{F^{\delta}_{a,k^\prime}} \Vert P_k v \Vert_{F^{\delta}_{a,k}}
\end{equation*}
In the resonant case it follows from Lemma \ref{lem:VanishingResonanceEstimate}
\begin{equation*}
II_{21} \lesssim T 2^{(1-a)n-k} \Vert P_n v \Vert_{F^{\delta}_{a,n}} \Vert P_{n^\prime} u_1 \Vert_{F^{\delta}_{a,n^\prime}} \Vert P_{k^\prime} u_1 \Vert_{F^{\delta}_{a,k^\prime}} \Vert P_k v \Vert_{F^{\delta}_{a,k}},
\end{equation*}
which is still sufficient.\\
For $II_{22}$ use Lemma \ref{lem:MultilinearL6StrichartzEstimate} to find
\begin{equation*}
\begin{split}
2^{(1-a)n-k} &\int_0^T dt \int P_n v P_{n_2} u_1 P_{n_3} u_1 P_k v \lesssim T 2^{(2-a+\delta)n} 2^{(1-a)n} 2^{-k} 2^{-(a+1)n/2} 2^{\varepsilon n} \\
&\Vert P_n v \Vert_{F^{\delta}_{a,n}} \Vert P_{n_2} u_1 \Vert_{F^{\delta}_{a,n_2}} \Vert P_{n_3} u_1 \Vert_{F^{\delta}_{a,n_3}} \Vert P_k v \Vert_{F^{\delta}_{a,k}}
\end{split}
\end{equation*}
and for $II_{23}$ we have to estimate
\begin{equation*}
2^{(1-a)n-k} \int_0^T dt \int P_n v P_{m_1} u_1 P_{m_2} u_1 P_k v, \quad |m_1-m_2| \leq 5, \; n \leq m_1-6
\end{equation*}
Here, we apply Lemma \ref{lem:HighLowLowHighMultilinearShorttimeEstimateFractionalBenjaminOno} to find
\begin{equation*}
II_{23} \lesssim T 2^{(1-a+\delta)m_1} 2^{(1-a)n -k} \Vert P_n v \Vert_{F^{\delta}_{a,n}} \Vert P_{m_1} u_1 \Vert_{F^{\delta}_{a,m_1}} \Vert P_{m_2} u_1 \Vert_{F^{\delta}_{a,m_2}} \Vert P_k v \Vert_{F^{\delta}_{a,k}} 
\end{equation*}
The estimate of $II_3$ is easier because the derivative hits a smaller frequency. But all frequencies can be comparable which leads to the expression
\begin{equation*}
2^{-an} \int_0^T dt \int P_n v P_{n^\prime} u_1 P_{m_1} v P_{m_2} u_1,
\end{equation*}
which can also be treated like above with Lemma \ref{lem:HighLowLowHighMultilinearShorttimeEstimateFractionalBenjaminOno} in the non-resonant case and Lemma \ref{lem:VanishingResonanceEstimate} in the resonant case.\\
In Case III we have to estimate
\begin{equation}
\label{eq:EnergyEstimateDifferenceCaseIIIFractionalBenjaminOno}
\int_0^T dt \int P_{n_1} u_1 P_{n_2} v P_{n_3} v 
\end{equation}
with $|n_i - n| \leq 10$ comparable.\\
The resonance is very favourable, and we find after integration by parts in time
\begin{equation*}
\begin{split}
\eqref{eq:EnergyEstimateDifferenceCaseIIIFractionalBenjaminOno} &= 2^{-(a+1)n} \left[ \int P_{n_1} u_1 P_{n_2} v P_{n_3} v \right]_{t=0}^T + 2^{-(a+1)n} \left( \int_0^T dt \int \partial_x P_{n_1} (u_1^2) P_{n_2} v P_{n_3} v \right. \\
&\left. + \int_0^T dt \int P_{n_1} u_1 \partial_x P_{n_2} (v u_1) P_{n_3} v + \int_0^T dt \int P_{n_1} u_1 P_{n_2} v \partial_x P_{n_3} (v u_1) \right) \\
&= B_{III}(0;T) + III_{1} + III_{2} + III_3 
\end{split}
\end{equation*}
Like above integration by parts in time is only carried out for high frequencies, which gives
\begin{equation*}
\sum_{n \geq m} \sum_{|n_i-n| \leq 10} 2^{-(a+1)n} \left[ \int P_{n_1} u_1 P_{n_2} v P_{n_3} v \right]_{t=0}^T \lesssim M^{-d} \Vert v \Vert^2_{F^{-1/2,\delta}_{a}(T)} \Vert u_1 \Vert_{F^{s,\delta}_a(T)}
\end{equation*}
and
\begin{equation*}
\sum_{1 \leq n \leq m} \sum_{|n_i-n| \leq 10} \int_0^T dt \int P_{n_1} u_1 P_{n_2} v P_{n_3} v \lesssim M^{c} T \Vert v \Vert^2_{F^{-1/2,\delta}_a(T)} \Vert u_1 \Vert_{F^{s,\delta}_a(T)}
\end{equation*}
Due to symmetry in the frequencies and multilinearity of the applied estimates we will only estimate $III_1$. We split $III_1 = III_{11} + III_{12} + III_{13}$ according to Littlewood-Paley decomposition. For $III_{11}$ we have to consider
\begin{equation*}
2^{-an} \int_0^T ds \int P_{n_1} u_1 P_k u_1 P_{n_2} v P_{n_3} v, \quad k \leq n-15,
\end{equation*}
and an application of Lemma \ref{lem:MultilinearL6StrichartzEstimate} gives
\begin{equation*}
III_{11} \lesssim T 2^{(2-a+\delta)n} 2^{\varepsilon n} 2^{-(a+1)n/2} 2^{-an} \Vert P_{n_1} u_1 \Vert_{F^{\delta}_{a,n_1}} \Vert P_k u_1 \Vert_{F^{\delta}_{a,k}} \Vert P_{n_2} v \Vert_{F^{\delta}_{a,n_2}} \Vert P_{n_3} v \Vert_{F^{\delta}_{a,n_3}}
\end{equation*}
For $III_{12}$ we have to estimate
\begin{equation*}
2^{-an} \int_0^T dt \int P_{n_1} u_1 P_{n_2} u_1 P_{n_3} v P_{n_4} v
\end{equation*}
with all frequencies comparable, i.e., $|n_i-n| \leq 15$. In the non-resonant case use Lemma \ref{lem:MultilinearL4StrichartzEstimate} and in the resonant case use Lemma \ref{lem:VanishingResonanceEstimate} to find
\begin{equation*}
III_{12} \lesssim T 2^{2(1-a)n} \Vert P_{n_1} u_1 \Vert_{F^{\delta}_{a,n_1}} \Vert P_{n_2} u_1 \Vert_{F^{\delta}_{a,n_2}} \Vert P_{n_3} v \Vert_{F^{\delta}_{a,n_3}} \Vert P_{n_4} v \Vert_{F^{\delta}_{a,n_4}}, \; |n_i-n| \leq 15
\end{equation*}
For $III_{13}$ we have to estimate
\begin{equation*}
2^{-an} \int_0^T dt \int P_{m_1} u_1 P_{m_2} u_1 P_{n_2} v P_{n_3} v, \quad |m_1-m_2| \leq 5, \quad n \leq m_1-6
\end{equation*}
An application of Lemma \ref{lem:HighLowLowHighMultilinearShorttimeEstimateFractionalBenjaminOno} yields
\begin{equation*}
III_{13} \lesssim 2^{-an} T 2^{(1-a+\delta)m_1} 2^{\varepsilon n} \Vert P_{m_1} u_1 \Vert_{F^{\delta}_{a,m_1}} \Vert P_{m_2} u_1 \Vert_{F^{\delta}_{a,m_2}} \Vert P_{n^\prime} v \Vert_{F^{\delta}_{a,n^\prime}} \Vert P_{n^{\prime \prime}} v \Vert_{F^{\delta}_{a,n^{\prime \prime}}}
\end{equation*}
This discloses the analysis of Case III.\\
In Case IV we are considering
\begin{equation}
\label{eq:EnergyEstimateDifferencesCaseIV}
\int_0^T dt \int P_n v (P_{m_1} u_1 P_{m_2} v), \quad |m_1-m_2| \leq 5, \; n \leq m_1-6
\end{equation}
An integration by parts in time yields
\begin{equation*}
\begin{split}
\eqref{eq:EnergyEstimateDifferencesCaseIV} &= 2^{-(am_1 + n)} \left[ \int P_n v P_{m_1} u_1 P_{m_2} v \right]_{t=0}^T \\
&+ 2^{-(am_1+n)} \left( \int_0^T dt \int \partial_x P_n (v u_1) P_{m_1} u_1 P_{m_2} v \right. \\
&\left. + \int_0^T dt \int P_n v \partial_x P_{m_1}( u_1^2) P_{m_2} v + \int_0^T dt \int P_n v P_{m_1} u_1 \partial_x P_{m_2} (v u_1) \right) \\
&= B_{IV}(0;T) + IV_1 + IV_2 + IV_3
\end{split}
\end{equation*}
Like above only the high frequencies are integrated by parts. For the corresponding boundary term we find by H\"older's inequality, Bernstein's inequality and Lemma \ref{lem:FsEmbedding} like for the previous boundary term $B_I$
\begin{equation*}
\begin{split}
B_{IV,M}(0;T) &= \sum_{m_1 \geq m} \sum_{n \leq m_1 - 6} \sum_{|m_1-m_2| \leq 5} 2^{-(am_1+n)} \left[ \int P_n v P_{m_1} u_1 P_{m_2} v \right]_{t=0}^T \\
&\lesssim M^{-d} \Vert v \Vert^2_{F^{-1/2,\delta}_a(T)} \Vert u_1 \Vert_{F^{s,\delta}_a(T)}
\end{split}
\end{equation*}
and for the low frequencies
\begin{equation*}
\sum_{m_1 \leq m} \sum_{n \leq m_1 -6} \sum_{|m_1-m_2| \leq 5} \int_0^T dt \int P_n v P_{m_1} u_1 P_{m_2} v \lesssim T M^{c} \Vert v \Vert^2_{F^{-1/2,\delta}_a(T)} \Vert u_1 \Vert_{F^{s,\delta}_a(T)}
\end{equation*}
Like above we split $IV_1 = IV_{11} + IV_{12} + IV_{13}$. To estimate $IV_{11}$ consider
\begin{equation*}
2^{-am_1} \int_0^T dt \int ( P_n v P_k u_1 + P_k v P_{n} u_1 ) P_{m_1} u_1 P_{m_2} v, \quad k \leq n-6
\end{equation*}
Since the second resonance does not vanish, $IV_{11}$ is amenable to Lemma \ref{lem:HighLowLowHighMultilinearShorttimeEstimateFractionalBenjaminOno} and we find
\begin{equation*}
\begin{split}
IV_{11} &\lesssim T 2^{(1-2a)m_1} 2^{\delta m_1} ( \Vert P_n v \Vert_{F^{\delta}_{a,n}} \Vert P_k u_1 \Vert_{F^{\delta}_{a,k}} + \Vert P_k v \Vert_{F^{\delta}_{a,k}} \Vert P_n u_1 \Vert_{F^{\delta}_{a,n}} ) \\
&\times \Vert P_{m_1} u_1 \Vert_{F^{\delta}_{a,m_1}} \Vert P_{m_2} v \Vert_{F^{\delta}_{a,m_2}} )
\end{split}
\end{equation*}
For $IV_{12}$ we can apply once more Lemma \ref{lem:HighLowLowHighMultilinearShorttimeEstimateFractionalBenjaminOno} to find
\begin{equation*}
IV_{12} \lesssim T 2^{(1-2a)m_1} 2^{\delta m_1} 2^{\varepsilon n} \Vert P_n v \Vert_{F^{\delta}_{a,n}} \Vert P_{n^\prime} u_1 \Vert_{F^{\delta}_{a,n^\prime}} \Vert P_{m_1} u_1 \Vert_{F^{\delta}_{a,m_1}} \Vert P_{m_2} v \Vert_{F^{\delta}_{a,m_2}}
\end{equation*}
and for $IV_{13}$ the only additional case arises when all frequencies are comparable in
\begin{equation*}
2^{-a m_1} \int_0^T dt \int P_{m_3} v P_{m_4} u_1 P_{m_1} u_1 P_{m_2} v, \quad \exists l: |m_i-l| \leq 10
\end{equation*}
In the non-resonant case use Lemma \ref{lem:MultilinearL4StrichartzEstimate} and in the resonant case Lemma \ref{lem:VanishingResonanceEstimate} to find
\begin{equation*}
IV_{13} \lesssim T 2^{2(1-a)l} 2^{\delta l} \Vert P_{m_1} u_1 \Vert_{F^{\delta}_{a,m_1}} \Vert P_{m_2} v \Vert_{F^{\delta}_{a,m_2}} \Vert P_{m_3} v \Vert_{F^{\delta}_{a,m_3}} \Vert P_{m_4} u_1 \Vert_{F^{\delta}_{a,m_4}}
\end{equation*}
We split $IV_2 = IV_{21} + IV_{22} + IV_{23}$. In case $IV_{21}$ we have to estimate
\begin{equation*}
2^{(1-a)m_1} 2^{-n} \int_0^T dt \int P_n v P_{m_1} u_1 P_k u_1 P_{m_2} v \quad (k,n \leq m_1-6, \; |m_1-m_2| \leq 5)
\end{equation*}
In the resonant case this expression is estimated by Lemma \ref{lem:VanishingResonanceEstimate} and in the non-resonant case use Lemma \ref{lem:HighLowLowHighMultilinearShorttimeEstimateFractionalBenjaminOno} to find
\begin{equation*}
IV_{21} \lesssim T 2^{(1-a) m_1} 2^{-n} \Vert P_{m_1} u_1 \Vert_{F^{\delta}_{a,n}} \Vert P_{m_2} v \Vert_{F^{\delta}_{a,m_2}} \Vert P_{n} v \Vert_{F^{\delta}_{a,n}} \Vert P_k u_1 \Vert_{F^{\delta}_{a,k}}
\end{equation*}
For $IV_{22}$ consider
\begin{equation*}
2^{(1-a)m_1-n} \int_0^T dt \int P_n v P_{m_1} u_1 P_{m_2} u_1 P_{m_3} v, \quad \exists m^\prime: n \leq m^\prime-10, |m_i - m^\prime| \leq 7
\end{equation*}
This we estimate by Lemma \ref{lem:MultilinearL6StrichartzEstimate} to find
\begin{equation*}
\begin{split}
IV_{22} &\lesssim T 2^{(2-a+\delta)m_1} 2^{-(a+1)m_1/2} 2^{(1-a)m_1} 2^{-n} \\
&\Vert P_n v \Vert_{F^{\delta}_{a,n}} \Vert P_{m_1} u_1 \Vert_{F^{\delta}_{a,m_1}} \Vert P_{m_2} u_1 \Vert_{F^{\delta}_{a,m_2}} \Vert P_{m_3} v \Vert_{F^{\delta}_{a,m_3}}
\end{split}
\end{equation*}
For $IV_{23}$ we have to estimate
\begin{equation*}
2^{(1-a)m_1} 2^{-n} \int_0^T dt \int P_n v P_{l_1} u_1 P_{l_2} u_1 P_{m_2} v, \quad n \leq m_2-3 \leq l_1-9, \quad |l_1-l_2| \leq 5
\end{equation*}
An application of Lemma \ref{lem:HighLowLowHighMultilinearShorttimeEstimateFractionalBenjaminOno} gives
\begin{equation*}
IV_{23} \lesssim T 2^{(1-a)m_1-n} 2^{(1-a)l_1} 2^{\delta l_1} \Vert P_n v \Vert_{F^{\delta}_{a,n}} \Vert P_{l_1} u_1 \Vert_{F^{\delta}_{a,l_1}} \Vert P_{l_2} u_1 \Vert_{F^{\delta}_{a,l_2}} \Vert P_{m_2} v \Vert_{F^{\delta}_{a,m_2}}
\end{equation*}
$IV_3$ is estimated like $IV_2$. This completes the proof of \eqref{eq:EnergyEstimateIDifferencesFractionalBenjaminOno}.\\[0.5cm]
%

In order to prove \eqref{eq:EnergyEstimateIIDifferencesFractionalBenjaminOno} we write by the fundamental theorem of calculus up to irrelevant factors
\begin{equation*}
\begin{split}
2^{2ns} \Vert P_n v(T) \Vert_{L^2}^2 &= 2^{2ns} \Vert P_n v(0) \Vert_{L^2}^2 + 2^{2ns} \int_0^T dt \int P_n v \partial_x P_n (v^2) \\
&+ 2^{2ns} \int_0^T dt \int P_n v \partial_x P_n (vu) = 2^{2ns} \Vert P_n v(0) \Vert_{L^2}^2 + 2^{2ns} (A+B)
\end{split}
\end{equation*}
where
\begin{equation*}
\begin{split}
A &= 2^{2ns} \int_0^T dt \int_{\Gamma_3} \chi_n^2(\xi_1) \hat{v}(\xi_1) (i \xi_1) \hat{v}(\xi_2) \hat{v}(\xi_3) d\Gamma_3 \\
 &= 2^{2ns} C \int_0^T dt \int_{\Gamma_3} d\Gamma_3 (\chi_n^2(\xi_1) \xi_1 + \chi_n^2(\xi_2) \xi_2 + \chi_n^2(\xi_3) \xi_3) \hat{v}(\xi_1) \hat{v}(\xi_2) \hat{v}(\xi_3) 
\end{split}
\end{equation*}
and after integration by parts in time we find modulo boundary terms
\begin{equation*}
\begin{split}
A &= \left[ \sum_{\substack{\xi_1 + \xi_2 + \xi_3 =0,\\ \xi_i \neq 0}} \frac{(\chi_n^2(\xi_1)\xi_1 + \chi_n^2(\xi_2) \xi_2 + \chi_n^2(\xi_3) \xi_3)}{\xi_1 |\xi_1|^a + \xi_2 |\xi_2|^a + \xi_3 |\xi_3|^a} \hat{v}(\xi_1) \hat{v}(\xi_2) \hat{v}(\xi_3) \right]_{t=0}^T \\
&+ \int_0^T dt \sum_{\substack{\xi_1 + \xi_2 + \xi_3 =0,\\ \xi_i \neq 0}} \frac{(\chi_n^2(\xi_1)\xi_1 + \chi_n^2(\xi_2) \xi_2 + \chi_n^2(\xi_3) \xi_3)}{\xi_1 |\xi_1|^a + \xi_2 |\xi_2|^a + \xi_3 |\xi_3|^a} \hat{v}(\xi_1) \hat{v}(\xi_2) \\
&\times \xi_3 \sum_{\substack{\xi_3 = \xi_{31} + \xi_{32}, \\ \xi_{3i} \neq 0}} \hat{v}(\xi_{31}) \hat{v}(\xi_{32}) \\
&+ \int_0^T dt \sum_{\substack{\xi_1 + \xi_2 + \xi_3 =0,\\ \xi_i \neq 0}} \frac{(\chi_n^2(\xi_1)\xi_1 + \chi_n^2(\xi_2) \xi_2 + \chi_n^2(\xi_3) \xi_3)}{\xi_1 |\xi_1|^a + \xi_2 |\xi_2|^a + \xi_3 |\xi_3|^a} \hat{v}(\xi_1) \hat{v}(\xi_2) \\
&\times \xi_3 \sum_{\substack{\xi_3 = \xi_{31} + \xi_{32}, \\ \xi_{3i} \neq 0}} \hat{v}(\xi_{31}) \hat{u}_2(\xi_{32}) \\
&= B_A(0;T) + A_1 + A_2
\end{split}
\end{equation*}
Set
\begin{equation*}
b_3(\xi_1,\xi_2,\xi_3) = \frac{\chi_n^2(\xi_1) \xi_1 + \chi_n^2(\xi_2) \xi_2 + \chi_n^2(\xi_3) \xi_3}{\xi_1 |\xi_1|^a + \xi_2 |\xi_2|^a + \xi_3 |\xi_3|^a}
\end{equation*}
A second symmetrization like in the proof of the energy estimates for solutions gives
\begin{equation*}
\begin{split}
A_1 &= C \int_0^T dt \int_{\Gamma_4} d\Gamma_4 b_3(\xi_1,\xi_2,\xi_{31}+\xi_{32}) \xi_3 \hat{v}(\xi_1) \hat{v}(\xi_2) \hat{v}(\xi_{31}) \hat{v}(\xi_{32})  \\
&= C \int_0^T dt \int_{\Gamma_4} d\Gamma_4 [b_3(\xi_1,\xi_2,\xi_{31}+\xi_{32}) - b_3(-\xi_{31},-\xi_{32},\xi_{31}+\xi_{32})] \\
&\times \xi_3 \hat{v}(\xi_1) \hat{v}(\xi_2) \hat{v}(\xi_{31}) \hat{v}(\xi_{32}) \\
&= C \int_0^T dt \int_{\Gamma_4} d\Gamma_4 b_4(\xi_1,\xi_2,\xi_{31},\xi_{32}) \hat{v}(\xi_1) \hat{v}(\xi_2) \hat{v}(\xi_{31}) \hat{v}(\xi_{32})
\end{split}
\end{equation*}
and the expression is estimated like in Lemma \ref{lem:RemainderEnergyEstimateSolutionsFractionalBenjaminOno}.\\
To estimate
\begin{equation*}
A_2 = \int_0^T dt \int_{\Gamma_4} d\Gamma_4 b_3(\xi_1,\xi_2,\xi_3) \xi_3 \hat{v}(\xi_1) \hat{v}(\xi_2) \hat{v}(\xi_{31}) \hat{u}_2(\xi_{32})
\end{equation*}
we conduct a Case-by-Case analysis plugging in Littlewood-Paley decomposition. For the interaction of $(v,v,v)$ before integration by parts in time we have to take into account the following cases:
\begin{enumerate}
\item[Case I:] $High \times Low \rightarrow High$ ($|\xi_1| \sim |\xi_3| \gg |\xi_2|$)
\item[Case II:] $High \times High \rightarrow High$ ($|\xi_1| \sim |\xi_2| \sim |\xi_3|$)
\item[Case III:] $High \times High \rightarrow Low$ ($|\xi_3| \ll |\xi_1| \sim |\xi_2|$)
\end{enumerate}
Here, we additionally plug in the possible frequency interactions for $(\xi_3,\xi_{31},\xi_{32})$ like $I = I_1 + I_2 + I_3$. For $I_1$ we have to estimate
\begin{equation*}
\begin{split}
I_1 &= 2^{2ns} 2^{(1-a)n} \left( \int_0^T dt \int P_n v P_k v \left( P_{n^\prime} v P_{k^\prime}u_2 + P_{k^\prime} v P_{n^\prime} u_2 \right) \right), \\
&\; k, k^\prime \leq n-6, \; |n-n^\prime| \leq 5
\end{split}
\end{equation*}
In the non-resonant case both expressions can be handled with Lemma \ref{lem:HighLowLowHighMultilinearShorttimeEstimateFractionalBenjaminOno} and in the resonant case Lemma \ref{lem:VanishingResonanceEstimate} yields
\begin{equation*}
\begin{split}
I_1 &\lesssim 2^{2ns} T 2^{(1-a)n}  \Vert P_n v \Vert_{F^{\delta}_{a,n}} \Vert P_k v \Vert_{F^{\delta}_{a,k}} \\
&\left( \Vert P_{n^\prime} v \Vert_{F^{\delta}_{a,n^\prime}} \Vert P_{k^\prime} u_2 \Vert_{F^{\delta}_{a,k^\prime}} + \Vert P_{k^\prime} v \Vert_{F^{\delta}_{a,k^\prime}} \Vert P_{n^\prime} u_2 \Vert_{F^{\delta}_{a,n^\prime}} \right) 
\end{split}
\end{equation*}
$I_2$ is amenable to Lemma \ref{lem:MultilinearL6StrichartzEstimate} which gives
\begin{equation*}
I_2 \lesssim 2^{2ns} T 2^{(2-a+\delta)n} 2^{(1-a)n} 2^{\varepsilon n} 2^{-(a+1)n/2} \Vert P_n v \Vert_{F^{\delta}_{a,n}} \Vert P_k v \Vert_{F^{\delta}_{a,k}} \Vert P_{n_2} v \Vert_{F^{\delta}_{a,n_2}} \Vert P_{n_3} u_2 \Vert_{F^{\delta}_{a,n_3}}, 
\end{equation*}
where $|n-n_i| \leq 5$, $k \leq n-10$.\\
For $I_3$ consider
\begin{equation*}
2^{2ns} 2^{(1-a)n} \int_0^T dt \int P_n v P_k v P_{l_1} v P_{l_2} u_2, \quad |l_1-l_2| \leq 5, \; n \leq l_1-6, \; k \leq n-6
\end{equation*}
Lemma \ref{lem:HighLowLowHighMultilinearShorttimeEstimateFractionalBenjaminOno} gives
\begin{equation*}
I_3 \lesssim 2^{2ns} T 2^{(1-a+\delta)l_1} 2^{(1-a)n} \Vert P_n v \Vert_{F^{\delta}_{a,n}} \Vert P_k v \Vert_{F^{\delta}_{a,k}} \Vert P_{l_1} v \Vert_{F^{\delta}_{a,l_1}} \Vert P_{l_2} u_2 \Vert_{F^{\delta}_{a,l_2}}
\end{equation*}
Consider Case II next. Split $II = II_1 + II_2 + II_3$. For $II_1$ we have to consider
\begin{equation*}
\begin{split}
&2^{2ns} 2^{(1-a)n} \left( \int_0^T dt \int P_{n_1} v P_{n_2} v P_{n_3} v P_k u_2 + \int_0^T dt \int P_{n_1} v P_{n_2} v P_k v P_{n_3} u_2 \right), \\
&\; |n_1-n_2| \leq 3, |n_1-n_3| \leq 3, k \leq n_1-6
\end{split}
\end{equation*}
This we estimate by Lemma \ref{lem:MultilinearL6StrichartzEstimate} to find
\begin{equation*}
\begin{split}
II_1 &\lesssim 2^{2ns} T 2^{(2-a+\delta)n} 2^{(1-a)n} 2^{-(a+1)n/2} \Vert P_{n_1} v \Vert_{F^{\delta}_{a,n_1}} \Vert P_{n_2} v \Vert_{F^{\delta}_{a,n_2}} \\
&\times ( \Vert P_{n_3} v \Vert_{F^{\delta}_{a,n_3}} \Vert P_k u_2 \Vert_{F^{\delta}_{a,k}} + \Vert P_k v \Vert_{F^{\delta}_{a,k}} \Vert P_{n_3} u_2 \Vert_{F^{\delta}_{a,n_3}})
\end{split}
\end{equation*}
For $II_2$ consider
\begin{equation*}
2^{2ns} 2^{(1-a)n} \int_0^T dt \int P_{n_1} v P_{n_2} v P_{n_3} v P_{n_4} u_2, \quad |n_1-n_i| \leq 10, \; i=2,3,4
\end{equation*}
This we estimate by Lemma \ref{lem:HighLowLowHighMultilinearShorttimeEstimateFractionalBenjaminOno} in the non-resonant case and by Lemma \ref{lem:VanishingResonanceEstimate} in the resonant case to find
\begin{equation*}
II_2 \lesssim 2^{2ns} T 2^{(3-2a)n} \Vert P_{n_1} v \Vert_{F^{\delta}_{a,n_1}} \Vert P_{n_2} v \Vert_{F^{\delta}_{a,n_2}} \Vert P_{n_3} v \Vert_{F^{\delta}_{a,n_3}} \Vert P_{n_4} u_2 \Vert_{F^{\delta}_{a,n_4}}
\end{equation*}
For $II_3$ we have to consider
\begin{equation*}
2^{2ns} 2^{(1-a)n_1} \int_0^T dt \int P_{n_1} v P_{n_2} v P_{l_1} v P_{l_2} u_2, \quad n_1 \leq l_1-10, \; |l_1-l_2| \leq 5, \; |n_1-n_2| \leq 5
\end{equation*}
This is amenable to Lemma \ref{lem:HighLowLowHighMultilinearShorttimeEstimateFractionalBenjaminOno} which yields the estimate
\begin{equation*}
II_3 \lesssim 2^{2ns} T 2^{(1-a+\delta)l_1} 2^{(1-a)n_1} \Vert P_{n_1} v \Vert_{F^{\delta}_{a,n_1}} \Vert P_{n_2} v \Vert_{P^{\delta}_{a,n_2}} \Vert P_{l_1} v \Vert_{F^{\delta}_{a,l_1}} \Vert P_{l_2} u_2 \Vert_{F^{\delta}_{a,l_2}}
\end{equation*}
We estimate $III = III_1 + III_2 + III_3$. For $III_1$ consider
\begin{equation*}
2^{2ns} 2^{(1-a)n} \int_0^T dt \int P_{n_1} v P_{n_2} v (P_k v P_{k^\prime} u_2 + P_{k^\prime} v P_k u_2),
\end{equation*}
where $|n_1-n_2| \leq 5$, $k \leq n_1-6$, $k^\prime \leq k-6$.\\
The expressions are amenable to Lemma \ref{lem:HighLowLowHighMultilinearShorttimeEstimateFractionalBenjaminOno} and we find
\begin{equation*}
\begin{split}
III_1 &\lesssim 2^{2ns} T 2^{2(1-a)n} 2^{\delta n} \Vert P_{n_1} v \Vert_{F^{\delta}_{a,n_1}} \Vert P_{n_2} v \Vert_{F^{\delta}_{a,n_2}} \\
&( \Vert P_k v \Vert_{F^{\delta}_{a,k}} \Vert P_{k^\prime} u_2 \Vert_{F^{\delta}_{a,k^\prime}} + \Vert P_{k^\prime} v \Vert_{F^{\delta}_{a,k^\prime}} \Vert P_k u_2 \Vert_{F^{\delta}_{a,k}} )
\end{split}
\end{equation*}
The same argument applies to $III_2$ because there can not be a resonant case, which gives
\begin{equation*}
\begin{split}
III_2 &\lesssim 2^{2ns} T 2^{(1-a)n_1} 2^{\delta n} \Vert P_{n_1} v \Vert_{F^{\delta}_{a,n_1}} \Vert P_{n_2} v \Vert_{F^{\delta}_{a,n_2}} \Vert P_{l_1} v \Vert_{F^{\delta}_{a,l_1}} \Vert P_{l_2} u_2 \Vert_{F^{\delta}_{a,l_2}} \\
&\; |n_1-n_2| \leq 5, \; |l_1-l_2| \leq 5, \; l_1 \leq n_1-10
\end{split}
\end{equation*}
For $III_3$ we have to consider
\begin{equation*}
\begin{split}
&2^{2ns} \int_0^T dt \int P_{n_1} v P_{n_2} v P_k (P_{l_1} v P_{l_2} u_2), \\
&|l_1-l_2| \leq 5, \; k \leq l_1-10, |n_1-n_2| \leq 5, k \leq n_1-10
\end{split}
\end{equation*}
If $|n_1-l_1| \geq 15$, we can argue like above. Otherwise, all frequencies are comparable and applying Lemma \ref{lem:MultilinearL4StrichartzEstimate} in the non-resonant case and Lemma \ref{lem:VanishingResonanceEstimate} in the resonant case to find
\begin{equation*}
III_3 \lesssim 2^{2ns} T 2^{(3-2a)n} \Vert P_{n_1} v \Vert_{F^{\delta}_{a,n_1}} \Vert P_{n_2} v \Vert_{F^{\delta}_{a,n_2}} \Vert P_{l_1} v \Vert_{F^{\delta}_{a,l_1}} \Vert P_{l_2} u_2 \Vert_{F^{\delta}_{a,l_2}}, |n_1-l_1| \leq 5.
\end{equation*}
For the estimate of $B$ we are again in the situation from the proof of \eqref{eq:EnergyEstimateIDifferencesFractionalBenjaminOno}. The only difference is that we do not have the extra smoothing from the $H^{-1/2}$-input regularity which leads to the shift in regularity.\\
We have the following cases:
\begin{enumerate}
\item[Case I:] $High \times Low \rightarrow High (v,u_2,v)$
\item[Case II:] $High \times Low \rightarrow High (v,v,u_2)$
\item[Case III:] $High \times High \rightarrow High$
\item[Case IV:] $High \times High \rightarrow Low (v,u_2,v)$
\end{enumerate}
To estimate the individual contributions we use exactly the same arguments from above. Hence, we will be brief.\\
In Case $I$ we integrate by parts to put the derivative on the lowest frequency from above to arrive at the expression
\begin{equation}
\label{eq:ReductionIEnergyEstimateL2FractionalBenjaminOno}
2^{2ns} 2^k \int_0^T dt \int dx P_n v P_{k} u_2 P_{n^\prime} v \quad (|n-n^\prime| \leq 5, \; k \leq n-6)
\end{equation}
Integration by parts in time gives modulo boundary terms and irrelevant factors
\begin{equation*}
\begin{split}
\eqref{eq:ReductionIEnergyEstimateL2FractionalBenjaminOno} - B_I(0;T) &= 2^{2ns} 2^{-an} \left( \int_0^T dt \int \partial_x P_n (v (v + u_2)) P_k u_1 P_{n^\prime} v \right. \\
&\left. + \int_0^T dt \int P_n v P_k \partial_x(u_2^2) P_{n^\prime} v \right) \\
&= I_1+I_2 ,\quad (|n-n^\prime| \leq 5, \; k,k^\prime \leq n-6)
\end{split}
\end{equation*}
The boundary terms are handled like in the proof of \eqref{eq:EnergyEstimateIDifferencesFractionalBenjaminOno}. We omit the estimates of the boundary terms in the following. Split $I_1 = I_{11} + I_{12} + I_{13}$. Using Lemma \ref{lem:HighLowLowHighMultilinearShorttimeEstimateFractionalBenjaminOno} in case of non-vanishing resonance and Lemma \ref{lem:VanishingResonanceEstimate} in case of vanishing second resonance
\begin{equation*}
\begin{split}
I_{11} &\lesssim 2^{2ns} T 2^{(1-a)n} \Vert P_{k} u_1 \Vert_{F^{\delta}_{a,k}} \Vert P_n v \Vert_{F^{\delta}_{a,n}} ( \Vert P_{n_1} v \Vert_{F^{\delta}_{a,n_1}} \Vert P_{k_1} u_2 \Vert_{F^{\delta}_{a,k}} \\
&+ \Vert P_{n_1} v \Vert_{F^{\delta}_{a,k_1}} \Vert P_{n_1} u_2 \Vert_{F_{a,n_1}^{\delta}} + \Vert P_{n_1} v \Vert_{F^{\delta}_{a,n_1}} \Vert P_{k_1} v \Vert_{F^{\delta}_{a,k_1}} ) 
\end{split}
\end{equation*}
For $I_{12}$ we find by the above argument
\begin{equation*}
I_{12} \lesssim 2^{2ns} T 2^{(1-a)n} \Vert P_{n_1} v \Vert_{F^{\delta}_{a,n_1}} ( \Vert P_{n_2} u_2 \Vert_{F^{\delta}_{a,n_2}} + \Vert P_{n_2} v \Vert_{F_{a,n_2}^{\delta}}) \Vert P_k u_2 \Vert_{F^{\delta}_{a,k}} \Vert P_{n_3} v \Vert_{F^{\delta}_{a,n_3}}
\end{equation*}
with $|n_i - n| \leq 5$, $k \leq n-10$.\\
Further,
\begin{equation*}
\begin{split}
I_{13} &\lesssim 2^{2ns} T 2^{(1-a)n} 2^{(1-a+\delta)m_1} \Vert P_{m_1} v \Vert_{F^{\delta}_{a,m_1}} \Vert P_k u_2 \Vert_{F^{\delta}_{a,k}} \Vert P_{k^\prime} v \Vert_{F^{\delta}_{a,k^\prime}} \\
&( \Vert P_{m_2} u_2 \Vert_{F^{\delta}_{a,m_2}} + \Vert P_{m_2} v \Vert_{F^\delta_{a,m_2}})
\end{split}
\end{equation*}
In case of $I_{23}$ the additional case of comparable frequency occurs
\begin{equation*}
2^{2ns} 2^{-an} 2^k \int_0^T dt \int P_n v P_{m_1} u_2 P_{m_2} u_2 P_{n^\prime} v, \quad |m_1-m_2| \leq 10, \; |m_1-n| \leq 10
\end{equation*}
and we find by Lemma \ref{lem:MultilinearL4StrichartzEstimate} or Lemma \ref{lem:VanishingResonanceEstimate}, respectively,
\begin{equation*}
I_{23} \lesssim 2^{2ns} T 2^{2(1-a)n} \Vert P_n v \Vert_{F^{\delta}_{a,n}} \Vert P_{m_1} u_2 \Vert_{F^{\delta}_{a,m_1}} \Vert P_{m_2} u_2 \Vert_{F^{\delta}_{a,m_2}} \Vert P_{n^\prime} v \Vert_{F^{\delta}_{a,n^\prime}}
\end{equation*}
In Case $II$ we have to estimate the expression
\begin{equation*}
2^{2ns} 2^n \int_0^T dt \int P_n v P_{n^\prime} u_2 P_k v \quad |n-n^\prime| \leq 5, \; k \leq n-10
\end{equation*}
This we integrate by parts in time to find 
\begin{equation*}
\begin{split}
II -B_{II}(0;T) &= 2^{2ns} 2^{(1-a)n-k} ( \int_0^T dt \int \partial_x P_n (v (v+ u_2)) P_{n^\prime} u_2 P_k v \\
&+ \int_0^T dt \int P_n v \partial_x P_{n^\prime} (u_2^2) P_k v + \int_0^T dt \int P_n v P_{n^\prime} u_2 \partial_x P_k (v(v+ u_2)) ) \\
&= II_1+II_2+II_3
\end{split}
\end{equation*}
By the above notation and arguments we find
\begin{equation*}
\begin{split}
II_{11} &\lesssim 2^{2ns} T 2^{(2-a)n-k} \Vert P_{n^\prime} u_2 \Vert_{F^{\delta}_{a,n^\prime}} \Vert P_k v \Vert_{F^{\delta}_{a,k}} ( \Vert P_n v \Vert_{F^{\delta}_{a,n}} \Vert P_{k^\prime} u_2 \Vert_{F^{\delta}_{a,k^\prime}} \\
&+ \Vert P_{k^\prime} v \Vert_{F^{\delta}_{a,k^\prime}} \Vert P_n u_2 \Vert_{F^{\delta}_{a,n}} + \Vert P_n v \Vert_{F^{\delta}_{a,n}} \Vert P_{k^\prime} v \Vert_{F^{\delta}_{a,k^\prime}}) \quad
k,k^\prime \leq n-10
\end{split}
\end{equation*}
with an improved estimate for $k \neq k^\prime$.\\
For $II_{12}$ estimate by Lemma \ref{lem:MultilinearL6StrichartzEstimate} 
\begin{equation*}
\begin{split}
&2^{2ns} 2^{(2-a)n-k} \int_0^T dt \int P_{n_1} v P_{n_2} u_2 P_{n^\prime} u_2 P_k v \\
&\lesssim 2^{2ns} T 2^{(2-a)n-k} 2^{-(a+1)/2} 2^{(\varepsilon + \delta)n} \Vert P_{n_1} v \Vert_{F^{\delta}_{a,n_1}} ( \Vert P_{n_2} u_2 \Vert_{F^{\delta}_{a,n_2}} + \Vert P_{n_2} v \Vert_{F^{\delta}_{a,n_2}} ) \\
&\Vert P_{n_3} u_2 \Vert_{F^{\delta}_{a,n^\prime}} \Vert P_k v \Vert_{F^{\delta}_{a,k}} \quad |n_i-n^\prime| \leq 5, \; k \leq n-10
\end{split}
\end{equation*}
For $II_{13}$ estimate by Lemma \ref{lem:HighLowLowHighMultilinearShorttimeEstimateFractionalBenjaminOno}
\begin{equation*}
\begin{split}
&2^{2ns} 2^{(2-a)n-k} \int_0^T dt \int P_{m_1} v ( P_{m_2} u_2 + P_{m_2} v) P_{n^\prime} u_2 P_k v \\
&\lesssim 2^{2ns} T 2^{(1-a+\delta)m_1} 2^{(2-a)n-k} \Vert P_{m_1} v \Vert_{F_{a,m_1}^{\delta}}  \Vert P_{n^\prime} u_2 \Vert_{F^{\delta}_{a,n^\prime}} \Vert P_k v \Vert_{F^{\delta}_{a,k}} \\
&( \Vert P_{m_2} u_2 \Vert_{F_{a,m_2}^{\delta}} + \Vert P_{m_2} v \Vert_{F^{\delta}_{a,m_2}} ), 
\end{split}
\end{equation*}
where $|m_1-m_2| \leq 5, \; n \leq m_1-6$.\\
For $II_{21}$ estimate
\begin{equation*}
2^{2ns} 2^{(2-a)n-k} \int_0^T dt \int P_n v P_{n^\prime} u_2 P_{k^\prime} u_2 P_k v, \quad |n-n^\prime| \leq 5, \; k,k^\prime \leq n-10
\end{equation*}
and it follows like above
\begin{equation*}
II_{21} \lesssim 2^{2ns} T 2^{(2-a)n-k} \Vert P_n v \Vert_{F^{\delta}_{a,n}} \Vert P_{n^\prime} u_2 \Vert_{F^{\delta}_{a,n^\prime}} \Vert P_{k^\prime} u_2 \Vert_{F^{\delta}_{a,k^\prime}} \Vert P_k v \Vert_{F^{\delta}_{a,k}}
\end{equation*}
with an improved estimate for $k \neq k^\prime$.\\
For $II_{22}$ we find by Lemma \ref{lem:MultilinearL6StrichartzEstimate}
\begin{equation*}
\begin{split}
II_{22} &\lesssim 2^{2ns} T 2^{(2-a+\delta)n} 2^{\varepsilon n} 2^{(2-a)n-k} 2^{-(a+1)/2} \\
&\Vert P_n v \Vert_{F^{\delta}_{a,n}} \Vert P_{n_2} u_2 \Vert_{F^{\delta}_{a,n_2}} \Vert P_{n_3} u_2 \Vert_{F^{\delta}_{a,n_3}} \Vert P_k v \Vert_{F^{\delta}_{a,k}},
\end{split}
\end{equation*}
where $|n-n_i| \leq 5$, $k \leq n-10$.\\
For $II_{23}$ we find by Lemma \ref{lem:HighLowLowHighMultilinearShorttimeEstimateFractionalBenjaminOno}
\begin{equation*}
II_{23} \lesssim 2^{2ns} T 2^{(2-a)n-k} 2^{(1-a+\delta)m_1} \Vert P_n v \Vert_{F^{\delta}_{a,n}} \Vert P_{m_1} u_2 \Vert_{F^{\delta}_{a,m_1}} \Vert P_{m_2} u_2 \Vert_{F^{\delta}_{a,m_2}} \Vert P_k v \Vert_{F^{\delta}_{a,k}}
\end{equation*}
For $II_3$ we can argue in case of separated frequencies like in $II_1$ or $II_2$ and the conclusion is easier because the derivative hits a low frequency. However, in case of comparable frequencies there is the additional case
\begin{equation}
\label{eq:EnergyEstimateReductionIIIFractionalBenjaminOno}
2^{2ns} 2^{(1-a)n} \int_0^T dt \int P_{n_1} v P_{n_2} u_2 P_{n_3} v ( P_{n_4} u_2 + P_{n_4} v) \quad \exists n: |n_i - n| \leq 10
\end{equation}
This is estimated by Lemma \ref{lem:MultilinearL4StrichartzEstimate} in case of non-vanishing resonance and \ref{lem:VanishingResonanceEstimate} otherwise to find
\begin{equation*}
\begin{split}
\eqref{eq:EnergyEstimateReductionIIIFractionalBenjaminOno} &\lesssim T 2^{(1-a)n} 2^{(2-a)n} \Vert P_{n_1} v \Vert_{F^{\delta}_{a,n_1}} \Vert P_{n_2} u_2 \Vert_{F^{\delta}_{a,n_2}} \Vert P_{n_3} v \Vert_{F^{\delta}_{a,n_3}} \\
&( \Vert P_{n_4} u_2 \Vert_{F^{\delta}_{a,n_4}} + \Vert P_{n_4} v \Vert_{F^{\delta}_{a,n_4}}) 
\end{split}
\end{equation*}
In Case III we find via the above arguments
\begin{align*}
III_{11} &\lesssim 2^{2ns} T 2^{(2-a+\delta)n} 2^{\varepsilon n} 2^{-(a+1)n/2} 2^{(1-a)n} \Vert P_{n_1} u_2 \Vert_{F^{\delta}_{a,n_1}} \Vert P_k u_2 \Vert_{F^{\delta}_{a,k}} \\
&\Vert P_{n_2} v \Vert_{F^{\delta}_{a,n_2}} \Vert P_{n_3} v \Vert_{F^{\delta}_{a,n_3}}, \text{ where } |n-n_i| \leq 10, \; k \leq n-15 \\
III_{12} &\lesssim 2^{2ns} T 2^{(3-2a)n} \Vert P_{n_1} u_2 \Vert_{F^{\delta}_{a,n_1}} \Vert P_{n_2} u_2 \Vert_{F^{\delta}_{a,n_2}} \Vert P_{n_3} v \Vert_{F^{\delta}_{a,n_3}} \Vert P_{n_4} v \Vert_{F^{\delta}_{a,n_4}}, \\
&\text{where }  |n_i-n| \leq 15\\
III_{13} &\lesssim T 2^{(1-a+\delta)m_1} 2^{(1-a)n} 2^{\varepsilon n} \Vert P_{m_1} u_2 \Vert_{F^{\delta}_{a,m_1}} \Vert P_{m_2} u_2 \Vert_{F^{\delta}_{a,m_2}} \Vert P_{n_1} v \Vert_{F^{\delta}_{a,n_1}} \Vert P_{n_2} v \Vert_{F^{\delta}_{a,n_2}}, \\
&\; \text{where } |m_1-m_2| \leq 5, \; n \leq m_1-10
\end{align*}
and due to symmetry and multilinearity the remaining cases are omitted.\\
In Case IV consider
\begin{equation*}
2^{2ns} 2^{n} \int_0^T dt \int P_n v (P_{m_1} u_2 P_{m_2} v), \quad |m_1-m_2| \leq 5, \; n \leq m_1-6
\end{equation*}
With the notation from above we find
\begin{align*}
IV_{11} &\lesssim 2^{2ns} T 2^{n} 2^{(1-2a)m_1} 2^{\delta m_1} \Vert P_{m_1} u_2 \Vert_{F^{\delta}_{a,m_1}} \Vert P_{m_2} v \Vert_{F^{\delta}_{a,m_2}} \\
&( \Vert P_n v \Vert_{F^{\delta}_{a,n}} ( \Vert P_k u_2 \Vert_{F^{\delta}_{a,k}} + \Vert P_k v \Vert_{F^{\delta}_{a,k}} ) + \Vert P_k v \Vert_{F^{\delta}_{a,k}} \Vert P_n u_2 \Vert_{F^{\delta}_{a,n}} )  \\
IV_{12} &\lesssim 2^{2ns} T 2^n 2^{(1-2a)m_1} 2^{\delta m_1} \Vert P_n v \Vert_{F^{\delta}_{a,n}} \Vert P_{m_1} u_2 \Vert_{F^{\delta}_{a,m_1}} \Vert P_{m_2} v \Vert_{F^{\delta}_{a,m_2}} \\
&( \Vert P_{n^\prime} u_2 \Vert_{F^{\delta}_{a,n^\prime}} + \Vert P_{n^\prime} v \Vert_{F^{\delta}_{a,n^\prime}} ) \\
IV_{13} &\lesssim 2^{2ns} T 2^{(2-3a)m_1} \Vert P_{m_2} v \Vert_{F^{\delta}_{a,m_2}} \Vert P_{m_3} v \Vert_{F^{\delta}_{a,m_3}} \Vert P_{m_4} u_2 \Vert_{F^{\delta}_{a,m_4}} \\
&( \Vert P_{m_1} u_2 \Vert_{F^{\delta}_{a,m_1}} + \Vert P_{m_1} v \Vert_{F^{\delta}_{a,m_1}} )
\end{align*}
For the other cases record
\begin{align*}
IV_{21} &\lesssim 2^{2ns} T 2^{(1-a+\delta)m_1} \Vert P_n v \Vert_{F^{\delta}_{a,n}} \Vert P_{m_1} u_2 \Vert_{F^{\delta}_{a,m_1}} \Vert P_k u_2 \Vert_{F^{\delta}_{a,k}} \Vert P_{m_2} v \Vert_{F^{\delta}_{a,m_2}} \\
IV_{22} &\lesssim 2^{2ns} T 2^{(2-a+\delta)m_1} \Vert P_n v \Vert_{F^{\delta}_{a,n}} \Vert P_{m_1} u_2 \Vert_{F^{\delta}_{a,m_1}} \Vert P_{m_2} u_2 \Vert_{F^{\delta}_{a,m_2}} \Vert P_{m_3} v \Vert_{F^{\delta}_{a,m_3}} \\
IV_{23} &\lesssim 2^{2ns} T 2^{(1-a)m_1} 2^{(1-a)l_1/2} 2^{\delta l_1} \Vert P_n v \Vert_{F^{\delta}_{a,n}} \Vert P_{l_1} u_2 \Vert_{F^{\delta}_{a,l_1}} \Vert P_{l_2} u_2 \Vert_{F^{\delta}_{a,l_2}} \Vert P_{m_2} v \Vert_{F^{\delta}_{a,m_2}}
\end{align*}
Case $IV_3$ is omitted due to multilinearity and symmetry.\\
All frequency localized estimates sum up to one of the below expressions choosing $\delta$ sufficiently small
\begin{align*}
&T \Vert v \Vert^3_{F_a^{s,\delta}(T)} \Vert u_2 \Vert_{F_a^{s,\delta}(T)} \\
&T \Vert v \Vert^2_{F_a^{s,\delta}(T)} \Vert u_2 \Vert_{F_a^{s,\delta}(T)}^2 \\
&T \Vert v \Vert_{F_a^{s,\delta}(T)} \Vert v \Vert_{F_a^{-1/2,\delta}(T)} \Vert u_2 \Vert_{F_a^{s+(2-a),\delta}(T)} \Vert u_2 \Vert_{F_a^{s,\delta}(T)}
\end{align*}
This finishes the proof of \eqref{eq:EnergyEstimateIIDifferencesFractionalBenjaminOno}.
\end{proof}

\end{document}